\documentclass{amsart}

\usepackage{etex}
\usepackage{amsmath}
\usepackage{amsthm}
\usepackage{amssymb}
\usepackage[frame,cmtip,arrow,matrix,line,graph,curve]{xy}
\usepackage{graphpap, color, paralist}
\usepackage{pstricks}
\usepackage[mathscr]{eucal}
\usepackage{mathabx}
\usepackage[pdftex,colorlinks,citecolor=blue,urlcolor=blue]{hyperref}
\usepackage{tikz}
\usetikzlibrary{calc,graphs,graphs.standard}
\usepackage{graphicx}
\usepackage{subcaption}
\usepackage{verbatim}
\usepackage{MnSymbol}
\usepackage{url}
\usepackage{array,multirow}

\usepackage{amsmath}
\usepackage{amsthm,amsfonts,amssymb,mathrsfs}
\usepackage{epic,eepic}
\usepackage{yfonts}
\usepackage{cleveref}
\usepackage{paralist,enumerate}

\newtheorem{theorem}{Theorem}[section]
\newtheorem{proposition}[theorem]{Proposition}
\newtheorem{lemma}[theorem]{Lemma}
\newtheorem{corollary}[theorem]{Corollary}

\theoremstyle{definition}
\newtheorem{definition}[theorem]{Definition}

\newtheorem{conjecture}[theorem]{Conjecture}
\newtheorem{remark}[theorem]{Remark}

\newcommand{\FF}{ \ensuremath{\mathbb{F}}}

\makeatletter
\def\moverlay{\mathpalette\mov@rlay}
\def\mov@rlay#1#2{\leavevmode\vtop{%
		\baselineskip\z@skip \lineskiplimit-\maxdimen
		\ialign{\hfil$\m@th#1##$\hfil\cr#2\crcr}}}
\newcommand{\charfusion}[3][\mathord]{
	#1{\ifx#1\mathop\vphantom{#2}\fi
		\mathpalette\mov@rlay{#2\cr#3}
	}
	\ifx#1\mathop\expandafter\displaylimits\fi}
\makeatother

\newcommand{\lk}{{\mathrm{lk}}}
\newcommand{\st}{\mathrm{st}}

\numberwithin{equation}{section}

\begin{document}
\title{Balanced triangulations on few vertices and an implementation of cross-flips}

\author[L. Venturello]{Lorenzo Venturello}
\email{lorenzo.venturello@uni-osnabrueck.de}
\address{
	Universit\"{a}t Osnabr\"{u}ck,
	Fakult\"{a}t f\"{u}r Mathematik,
	Albrechtstra\ss e 28a,
	49076 Osnabr\"{u}ck, GERMANY
}

\date{\today}

\thanks{
	The author was supported by the German Research Council DFG GRK-1916.
}

\begin{abstract}
	A $d$-dimensional simplicial complex is balanced if the underlying graph is $(d+1)$-colorable. We present an implementation of cross-flips, a set of local moves introduced by Izmestiev, Klee and Novik which connect any two PL-homeomorphic balanced combinatorial manifolds. As a result we exhibit a vertex minimal balanced triangulation of the real projective plane, of the dunce hat and of the real projective space, as well as several balanced triangulations of surfaces and 3-manifolds on few vertices. In particular we construct small balanced triangulations of the 3-sphere that are non-shellable and shellable but not vertex decomposable.
\end{abstract}	
	
\maketitle

\thispagestyle{empty}
\section{Introduction}
\noindent
The study of the number of faces in each dimension that a triangulation of a manifold $M$ can have is a very classical and hard problem in combinatorial topology. Even for the case of triangulated spheres the characterization of these numbers is a celebrated unsolved problem, known as the $g$-conjecture. On an apparently simpler level one can ask the following question: what is the minimum number of vertices needed to triangulate a manifold $M$? Again the picture is far from being complete. An interesting tool to approach these and other kind of problems are \emph{bistellar flips}, a finite set of local moves which preserves the PL-homeomorphism type, and suffices to connect any two \emph{combinatorial triangulations} of a given manifold (equivalently, triangulations of PL-manifolds).
Bj\"{o}rner and Lutz \cite{BjLu} designed a computer program called BISTELLAR, employing bistellar flips in order to obtain triangulations on few vertices and heuristically recognize the homeomorphism type. This tool led to a significant number of small or even vertex-minimal triangulations which are listed in The Manifold Page \cite{Lu}, along with many other interesting examples (see also \cite{BeLuRandom} and \cite{BjLu}).\\
In this article we focus on the family of \emph{balanced} simplicial complexes, i.e., $d$-dimensional complexes whose underlying graph is $(d+1)$-colorable, in the classical graph theoretic sense. Many questions and results in face enumeration have balanced analogs (see for instance \cite{IKN, JKM, Juhnke:Murai:Novik:Sawaske, KN}. In particular we can ask the following question: what is the minimum number of vertices that a balanced triangulation of a manifold $M$ can have? 
Izmestiev, Klee and Novik introduced a finite set of local moves called \emph{cross-flips}, which preserves balancedness, the PL-homeomorphism type, and suffices to connect any two balanced combinatorial triangulations of a manifold. We provide primitive computer program implemented in Sage \cite{sagemath} to search the set of balanced triangulations of a manifold and we obtain the following results:
\begin{itemize}
	\item We find balanced triangulations of surfaces on few vertices. In particular we describe the unique vertex minimal balanced triangulation of $\mathbb{R}\mathbf{P}^{2}$ on 9 vertices.
	\item We find a balanced triangulation of the dunce hat on $11$ vertices. \Cref{section_dunce} is devoted to the proof of its vertex-minimality.
	\item In \Cref{dimension_3} we discuss balanced triangulations of $3$-manifolds on few vertices. In particular we exhibit a vertex minimal balanced triangulation of $\mathbb{R}\mathbf{P}^{3}$ on 16 vertices with interesting symmetries, and triangulations of the connected sums $(S^2\times S^1)^{\#2}$ and $(S^2\dtimes S^1)^{\#2}$ that belong to the balanced Walkup class.
	\item Finally in \Cref{section_decomposition} we construct balanced 3-spheres on few vertices that are non-shellable and shellable but not vertex decomposable, using results in knot theory.   
\end{itemize}
The source code and the list of facets of all simplicial complexes appearing in this paper are made available in \cite{GitRepo}.

\section{Preliminaries}
\noindent
An \emph{abstract simplicial complex} $\Delta$ on $\left[ n\right] $ is a collection of elements of $2^{\left[ n\right] }$ that is closed under inclusion. The elements $F\in\Delta$ are called \emph{faces}, and those that are maximal w.r.t. inclusion are called \emph{facets}. A simplicial complex is uniquely determined by its facets: for elements $F_i\in2^{\left[ n\right] }$ we define the complex \emph{generated} by $\left\lbrace F_1,\dots,F_m\right\rbrace$ as
$$\left\langle F_1,\dots,F_m \right\rangle:=\left\lbrace F\in2^{\left[ n\right] } : F\subseteq F_i, \text{ for some }i=1,\dots, m  \right\rbrace.$$
A \emph{subcomplex} of $\Delta$ is any simplicial complex $\Phi\subseteq\Delta$, and such a subcomplex $\Phi$ on $V\subseteq[n]$ is \emph{induced} if for every subset $G\subseteq 2^V$ such that $G\in\Delta$ then $G\in\Phi$. 
The \emph{dimension} of a face $F$ is the integer $\dim(F):=\left|F \right|-1$, and the dimension of a simplicial complex is the maximal dimension of its facets. $0$-dimensional and $1$-dimensional faces are called \emph{vertices} and \emph{edges}, and faces that are maximal w.r.t. inclusion are called \emph{facets}. A complex whose facets all have the same dimension is called \emph{pure}. We denote with $f_i(\Delta)$ the number of faces of $\Delta$ of dimension $i$, and we collect them together in the \emph{$f$-vector} $f(\Delta)=(f_{-1}(\Delta),f_0(\Delta),\dots,f_{\dim(\Delta)}(\Delta))$.  Given two simplicial complexes $\Delta$ and $\Gamma$, their \emph{join} is defined to be 
$$\Delta*\Gamma:=\left\lbrace F\cup G: F\in\Delta, G\in\Gamma \right\rbrace.$$
In particular, for two vertices $i,j\notin\Delta$ the operations $\Delta*\left\langle \left\lbrace i\right\rbrace  \right\rangle$ and $\Delta*\left\langle \left\lbrace i\right\rbrace ,\left\lbrace j\right\rbrace  \right\rangle$ are respectively the \emph{cone} and the \emph{suspension} over $\Delta$. To every face $F\in\Delta$ we associate  the two simplicial complexes
$$\lk_{\Delta}(F):=\left\lbrace G\in\Delta : F\cup G\in\Delta, F\cap G=\emptyset \right\rbrace,$$
and
$$\st_{\Delta}(F):=\left\langle F\right\rangle *\lk_{\Delta}(F),$$ 
called the \emph{link} and the \emph{star} of $\Delta$ at $F$, which describe the "local" properties of $\Delta$. There is a canonical way to associate to an abstract simplicial complex $\Delta$ a topological space, denoted by $\left| \Delta\right|$, and via this correspondence the terminology and the operations defined above represent discrete analogs of well known tools from classical topology. On the other hand, given a topological space $X$, a \emph{triangulation} of $X$ is any simplicial complex $\Delta$ such that $\left|\Delta \right|\cong X$. For example the complex $\partial\Delta_{d+1}:=\left\langle [d+2]\setminus\left\lbrace i\right\rbrace , i\in[d+2]\right\rangle $ is a standard triangulation of the $d$-sphere $S^{d}$. More specifically we define the following. 
\begin{definition}
	A pure $d$-dimensional simplicial complex $\Delta$ is a \emph{combinatorial $d$-sphere} if $\left| \Delta\right| $ is PL-homeomor\-phic to $\left| \partial\Delta_{d+1}\right|$. A pure connected $d$-dimensional simplicial complex is a (closed) \emph{combinatorial $d$-manifold} if the link of each vertex is a combinatorial $(d-1)$-sphere. 
\end{definition}
\noindent
Even though there is a subtle difference between the class of combinatorial manifolds and that of simplicial complexes homeomorphic to a manifold (often called simplicial manifolds), they are known to coincide for $d\leq 3$ (see \Cref{dimension_3_sub} for more details in the case $d>3$). For a fixed field $\FF$ a relaxation of the above definitions is given by the class of \emph{$\FF$-homology $d$-manifolds}, that is pure $d$-dimensional simplicial complexes such that $\widetilde{H}_i(\lk_{\Delta}(F);\FF)\cong \widetilde{H}_i(S^{d-\dim(F)-1};\FF)$ holds for every nonempty face $F$ and for every $i\geq 0$ (i.e., $\lk_{\Delta}(F)$ is a \emph{$\FF$-homology $(d-\dim(F)-1)$-sphere}). \\
In this paper we study a family of complexes with an additional combinatorial property, introduced by Stanley in \cite{St79}.
\begin{definition}
	A $d$-dimensional simplicial complex $\Delta$ on $[n]$ is \emph{balanced} if there is a map $\kappa:[n]\longrightarrow[d+1]$, such that $\kappa(i)\neq\kappa(j)$ for every $\left\lbrace i,j \right\rbrace\in\Delta$. 
\end{definition}
\noindent
In words $\Delta$ is balanced if the graph given by its vertices and edges is $(d+1)$-colorable in the classical graph theoretic sense, therefore we often refer to the elements in $[d+1]$ as colors, and to the preimage of a color as \emph{color class}. Although a priori the map $\kappa$ is part of the data defining a balanced complex, in all the examples considered in this paper $\kappa$ is unique up to permutations of the colors. We turn our attention to balanced triangulations of interesting topological spaces. As a guide example we consider the $d$-dimensional complex
\begin{equation}\label{crossp}
\partial\mathcal{C}_{d+1}:=\left\langle \left\lbrace 0\right\rbrace ,\left\lbrace v_0\right\rbrace \right\rangle*\dots*\left\langle \left\lbrace d\right\rbrace ,\left\lbrace v_d\right\rbrace \right\rangle,
\end{equation} 
on the set $\left\lbrace 0,\dots,d,v_0,\dots,v_d\right\rbrace $. This is indeed a balanced vertex minimal triangulation of $S^d$, and it is in particular isomorphic to the boundary of the $(d+1)$-dimensional \emph{cross-polytope}.
In general it is possible to turn any triangulation $\Delta$ of a topological space into a balanced one by considering its \emph{barycentric subdivision} $\text{Bd}(\Delta)$, defined as
$$\text{Bd}(\Delta):=\left\lbrace \left\lbrace v_{F_1},\dots,v_{F_m} \right\rbrace : F_1\subsetneq\dots\subsetneq F_m, F_i\in\Delta \right\rbrace.$$ 
Indeed more generally for the order complex of a ranked poset, the rank function gives a coloring which provides balancedness (see \cite{St79}).  
Among the many results on face enumeration that have been recently proved to have balanced analogs we focus on a work of Izmestiev, Klee and Novik \cite{IKN}, which specializes the theory of bistellar flips to the balanced setting. In their work the following operation preserving balancedness is defined.
\begin{definition}\label{cross-flips}
	Let $\Delta$ be a pure $d$-dimensional simplicial complexes and let $\Phi\subseteq\Delta$ be an induced subcomplex that is PL-homeomorphic to a $d$-ball and that is isomorphic to a subcomplex of $\partial\mathcal{C}_{d+1}$. The operation
	$$\Delta\longmapsto\chi_{\Phi}(\Delta):=\Delta\setminus\Phi \cup \left( \partial\mathcal{C}_{d+1}\setminus\Phi \right)$$
	is called a \emph{cross-flip} on $\Delta$.
\end{definition}
\begin{figure}[h]
	\centering
	\includegraphics[scale=0.75]{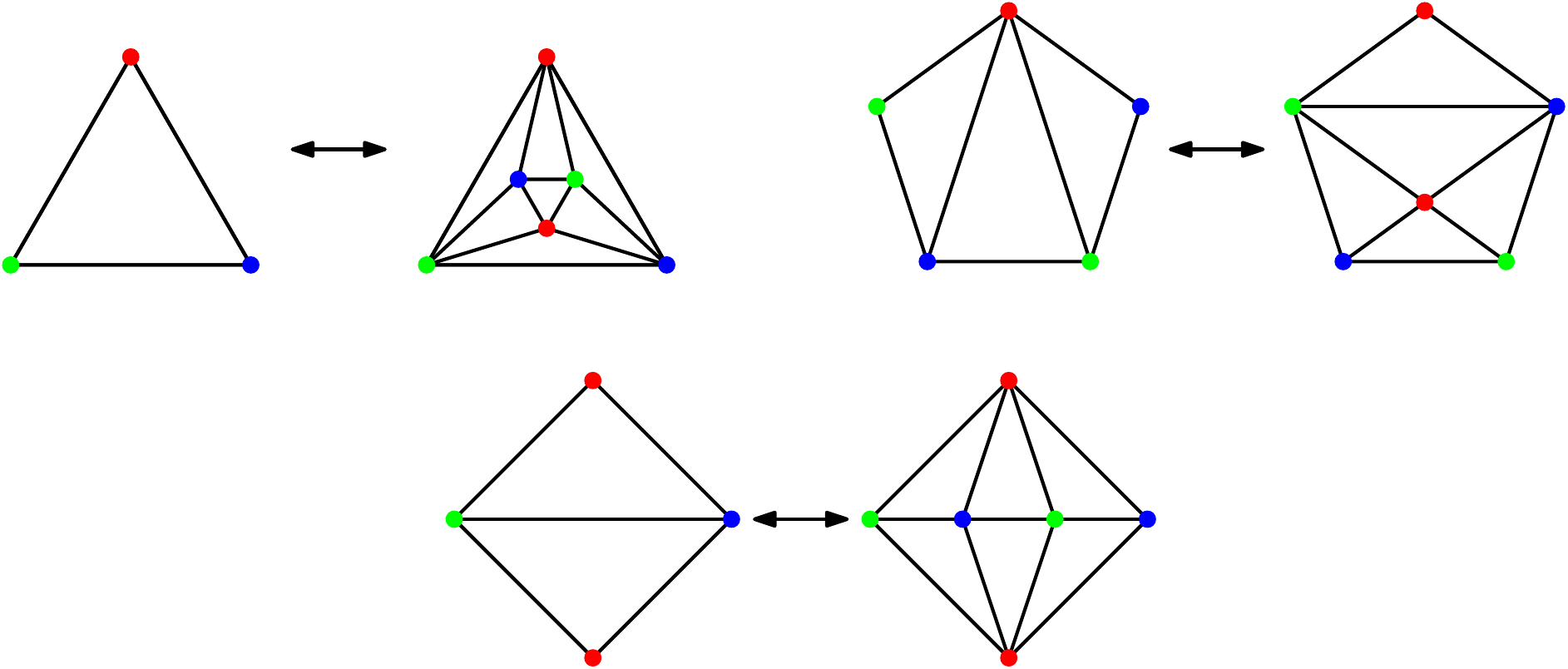}
	\caption{All non-trivial basic cross-flips for $d=2$.}
	\label{2_dim}
\end{figure} 
\noindent
In \cite{IKN} the authors require the subcomplexes $\Phi$ and $\partial\mathcal{C}_{d+1}\setminus\Phi$ to be \emph{shellable} (see \Cref{section_decomposition} for a definition), but since all the subcomplexes we consider in our implementation satisfy this condition we do not include it in \Cref{cross-flips}.
We now describe an interesting family of subcomplexes of $\partial\mathcal{C}_{d+1}$: for $0\leq i \leq d+1$ define
$$\Phi_i:=\begin{cases}
\left\langle \left\lbrace v_{0} \right\rbrace \right\rangle * \left\langle \left\lbrace i+1\right\rbrace ,\left\lbrace v_{i+1}\right\rbrace \right\rangle *\dots* \left\langle \left\lbrace d\right\rbrace ,\left\lbrace v_{d}\right\rbrace \right\rangle & \text{for } i=0 \\
\left\langle \left\lbrace 0,\dots,i-1,v_{i} \right\rbrace \right\rangle * \left\langle \left\lbrace i+1\right\rbrace ,\left\lbrace v_{i+1}\right\rbrace \right\rangle *\dots* \left\langle \left\lbrace d\right\rbrace ,\left\lbrace v_{d}\right\rbrace \right\rangle & \text{for } 1\leq i \leq d \\
\left\langle \left\lbrace 0,\dots,d\right\rbrace \right\rangle
& \text{for } i = d+1
\end{cases},$$
and let $\Phi_I:=\bigcup_{i\in I}\Phi_i$, for every $I\subseteq [d+1]$. It is not hard to see that those complexes are indeed shellable subcomplexes of the boundary of the $(d+1)$-dimensional cross-polytope in \ref{crossp}. A cross-flip replacing a subcomplex $\Phi_I$ with its complement $\Phi_J$ (note that this family is closed under taking complement w.r.t. $\partial\mathcal{C}_{d+1}$) is called a \emph{basic cross-flip}.
The basic cross-flip replacing $\Phi_{\left\lbrace 0\right\rbrace} $ with $\partial\mathcal{C}_{d+1}\setminus \Phi_{\left\lbrace 0\right\rbrace} \cong \Phi_{\left\lbrace 0\right\rbrace}$ is referred to as \emph{trivial flip}, because it clearly does not affect the combinatorics, and every non-trivial basic cross-flips either increases or decreases the number of vertices. We refer to the former as \emph{up-flips} and to the latter as \emph{down-flips}. Moreover two distinct sets $I\neq J\subseteq[d+1]$ might lead to isomorphic subcomplexes $\Phi_I\cong\Phi_J$, and certain basic cross-flips can be generated (i.e. written as combination) by some others. As an example, the flips in the second line  of \Cref{2_dim} (we count the arrows separately) can be obtained via a combination of the four moves in the first row. These issues, as well as a description of the possible $f$-vectors of the complexes $\Phi_I$, have been studied in \cite{2018arXiv180406270J}.
\begin{theorem}[\cite{2018arXiv180406270J}]\label{reducedCrossFlips}
	There are precisely $2^{d+1}-2$ non isomorphic non-trivial basic cross-flips in dimension $d$. Moreover $2^{d}$ of them suffice to generate them all. 
\end{theorem}
\noindent
The case of surfaces has been also studied in \cite{MuSu}. The interest in cross-flips, and in particular in basic cross-flips, is due to the following result.
\begin{theorem}[\cite{IKN}]\label{cross-flipsIKN}
	Let $\Delta$ and $\Gamma$ be balanced combinatorial $d$-manifolds. Then the following conditions are equivalent:
	\begin{itemize}
		\item $\left| \Delta\right| $ and $\left| \Gamma\right| $ are PL-homeomorphic;
		\item $\Delta$ and $\Gamma$ are connected by a sequence of cross-flips;
		\item $\Delta$ and $\Gamma$ are connected by a sequence of basic cross-flips.
	\end{itemize} 
\end{theorem}
\noindent	 
Essentially \Cref{cross-flipsIKN} states that any two balanced PL-homeomorphic combinatorial manifolds can be transformed one into the other by a sequence of a finite number of flips. This serves as a motivation to develop an implementation of this moves, as it was done in the setting of bistellar flips by Bj\"{o}rner and Lutz in \cite{BjLu} with BISTELLAR. In particular our goal is to find balanced triangulations of a given manifold on few vertices, since taking barycentric subdivision typically leads to large complexes.
\begin{remark}
	We conclude this section by offering a way to visualize the results above. Consider a graph whose vertices are all the balanced combinatorial triangulation of a certain manifold $M$, and whose edges are basic cross-flips. We call this graph the \emph{cross-flip graph} of $M$, and observe that \Cref{cross-flipsIKN} states that this graph is connected. Furthermore we can associate to $M$ another connected graph on the same vertex set, but with edges the sufficient flips guaranteed in \Cref{reducedCrossFlips}. We call this graph the \emph{reduced cross-flip graph}. \Cref{big_graph} shows a plot of a subgraph of the reduced cross-flip graph, displayed by ranking the (all non-isomorphic) complexes according to the number of vertices (see the numbers on the left). Here we stop performing up-flips on a sphere $\Delta$ if $f_0(\Delta)\geq 14$. Note that there is no guarantee of enumerating \emph{all} the balanced spheres in this way: as an example all the spheres in \Cref{big_graph} with $f_0(\Delta)=16$ satisfy $f_1(\Delta)\leq 72$, whereas in \cite{Zh17} a balanced $3$-sphere on $16$ vertices with $96$ edges is constructed.
\end{remark}   
\begin{table}[h!]
	\[\begin{array}{l|l|l|l}
	J&f(\Phi_J) &K & f(\partial\mathcal{C}_{d+1}\setminus\Phi_J)=f(\Phi_K) \\ \hline
	\left[3\right]&(1,4,6,4,1)&[0,1,2,3]&(1,8,24,32,15)\\
	\left[ 2\right] &(1,5,9,7,2)&[0,1,2]&(1,8,24,31,14)\\
	\left[2,3\right]&(1,6,12,10,3)&[0,1,3]&(1,8,24,30,13) \\
	\left[1\right]&(1,6,13,12,4)&[0,1]&(1,8,23,28,12)\\
	\left[1,3\right]&(1,7,16,15,5)&\left[0,2,3\right]&(1,8,23,27,11)\\
	\left[1,2\right]&(1,7,17,17,6)&\left[0,2\right]&(1,8,22,25,10)\\
	\left[1,2,3\right]&(1,7,18,19,7)&\left[0,3\right]&(1,8,21,23,9)\\
	\left[0\right]&(1,7,18,20,8)&\left[0\right]&(1,7,18,20,8)
	\end{array}
	\]
	\caption{$f$-vectors of $3$-dimensional basic cross-flips.}
	\label{tab:table1}
\end{table}
\begin{figure}[h]
	\centering
	\includegraphics[scale=0.74]{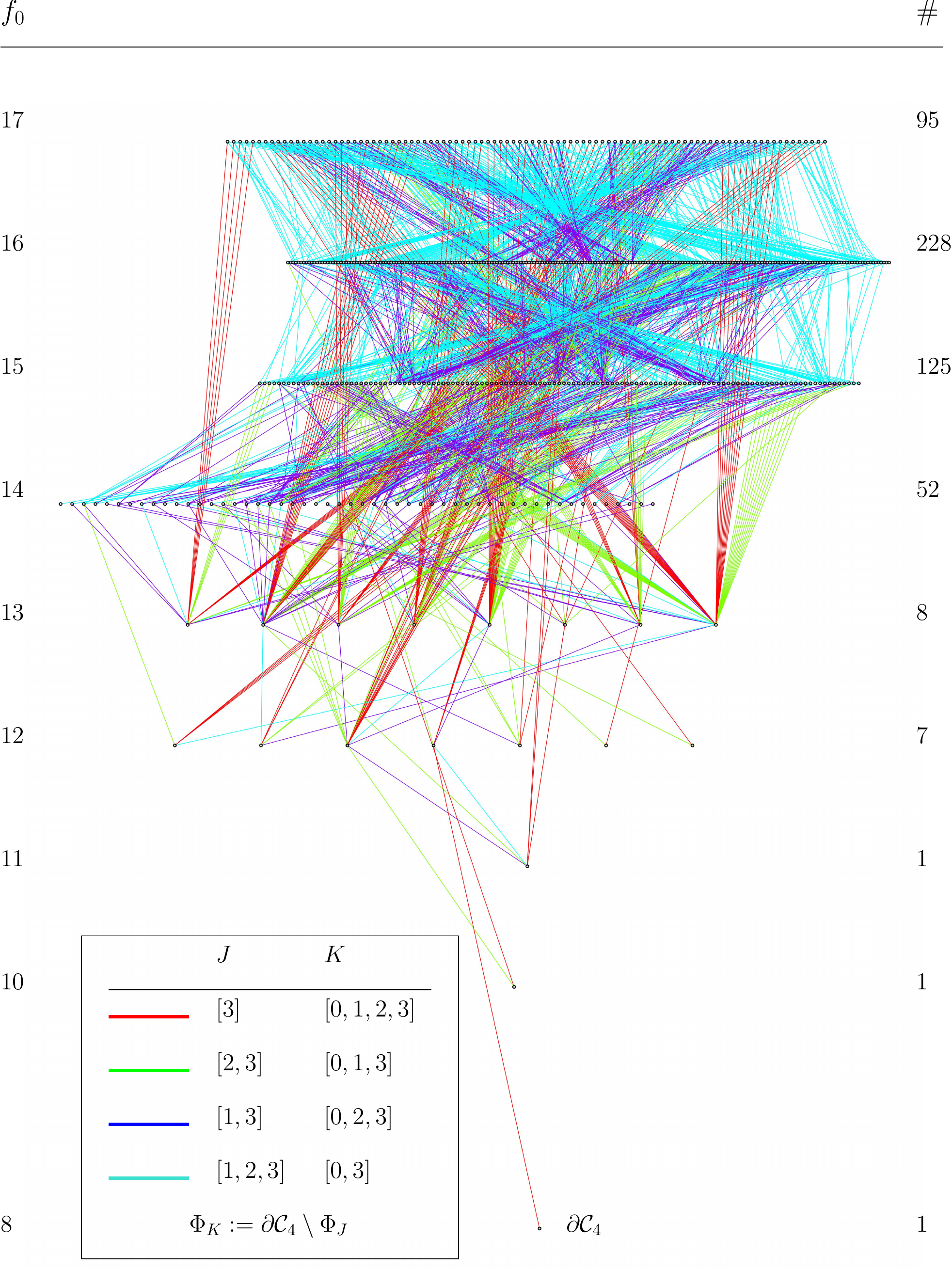}
	\caption{The subgraph of the reduced cross-flips graph of $S^3$ obtained applying the 4 sufficient cross-flips $\Delta\longmapsto\chi_{\Phi_{J}}(\Delta)$, starting from $\partial\mathcal{C}_4$. Edges represent flips, and different colors correspond to different flips. Up-flips are performed for $f_0<14$.}
	\label{big_graph}
\end{figure} 
\section{The implementation}
\noindent
The main purpose of our implementation is to obtain small, possibly vertex minimal, balanced triangulations of surfaces and 3-manifolds. To achieve this we start from the barycentric subdivision of a non-balanced triangulation, many of which can be found in \cite{Lu}, and reduce them using cross-flips. We first establish some notations: a vertex $v\in\Delta$ is called \emph{removable} if there exists a down flip $\chi_{\Phi}$ such that $v\notin\chi_{\Phi}(\Delta)$. A balanced simplicial complex without removable vertices is called \emph{irreducible}. In \Cref{big_graph} irreducible triangulations of $S^3$ can be visualized as vertices not connected with any lower vertex.
\begin{remark}
	While a vertex minimal balanced triangulation is clearly irreducible, the converse is not true. Indeed irreducible triangulations are quite many, and they can have a large vertex set, as shown in \Cref{bd_irreducible}.
\end{remark} 
\begin{lemma}
	Let $\Delta$ be a pure $d$-dimensional balanced simplicial complex. If a vertex $v\in \Delta$ is removable then $f_0(\lk_{\Delta}(v))=2d$. 
\end{lemma}
\begin{proof}
	If the vertex $v$ is removable then there exists an induced subcomplex $\Gamma\subseteq\Delta$ that is isomorphic to a subcomplex of $\partial\mathcal{C}_{d+1}$, such that $\Gamma$ is a $d$-ball and $v$ is in the interior of $\Gamma$, because vertices on the boundary are preserved. Since the link of a vertex in the interior of a balanced $d$-ball is a balanced $(d-1)$-sphere, and the only such subcomplex of $\partial\mathcal{C}_{d+1}$ is isomorphic to $\partial\mathcal{C}_{d}$ it follows that $\lk_{\Delta}{(v)}\cong\partial\mathcal{C}_{d}$, hence $f_0(\lk_{\Delta}{(v)})=2d$. 
\end{proof}
\begin{corollary}\label{bd_irreducible}
	Let $\Delta$ be a combinatorial $d$-manifold, with $d\geq 3$. Then the barycentric subdivision $\text{Bd}(\Delta)$ is irreducible.
\end{corollary}
\begin{proof}
	For every vertex $v_F\in \text{Bd}(\Delta)$, corresponding to a $k$-face $F\in\Delta$ we have $\lk_{\text{Bd}(\Delta)}(v_F)\cong\text{Bd}(\partial\Delta_{k})*\text{Bd}(\lk_{\Delta}(F))$. Moreover, since $\lk_{\text{Bd}(\Delta)}(v_F)$ is a combinatorial $(d-k-1)$-sphere, it has at least $f_i(\partial\Delta_{d-k})$ $i$-faces. Hence
	\begin{align*}
	f_0(\lk_{\text{Bd}(\Delta)}(v_F))&=2^{k+1}-2+\sum_{i=0}^{d-k-1}f_i(\lk_{\Delta}(F))\\
	&\geq 2^{k+1}-2+\sum_{i=0}^{d-k-1}f_i(\partial\Delta_{d-k})\\
	&= 2^{d-k+1}+2^{k+1}-4.
	\end{align*}
	For a fixed $d$ the last expression is minimized when $k=\frac{d}{2}$, and in that case we have $f_0(\lk_{\text{Bd}(\Delta)}(v_F))\geq 4\left( 2^{\frac{d}{2}}-1\right)$, which is strictly larger than $2d$, for $d\geq 3$. 
\end{proof}
\noindent
Note that \Cref{bd_irreducible} holds for more general classes, e.g., homology manifolds, for which the inequality $f_i(\lk_{\Delta}(F))\geq f_i(\partial\Delta_{d-k})$ still holds for every $k$-face $F$ and for every $i$.
The computation above also shows that for $d=2$ every vertex $v_F$ arising from the subdivision of an edge $F$ has degree $4$, hence it is potentially removable. Since the barycentric subdivision is the standard way of turning any triangulation into a balanced triangulation of the same space, for $d\geq 3$ the result above represents a bad news. Indeed \Cref{bd_irreducible} states that to reduce such a subdivision we are forced to start with some up-flips and to increase the number of vertices, which for the case of barycentric subdivisions is typically quite large.
Our code presents two main challenges:
\begin{itemize}
	\item List all the flippable subcomplexes of any type;
	\item Decide which type of move to apply and which subcomplex to flip. 
\end{itemize} 
Already in dimension 1 the problem of deciding if a fixed complex has a subcomplex isomorphic to a given one, known as the subgraph isomorphism problem, is NP-complete. However since graphs are computationally well studied it is convenient to reduce the problem to the one dimensional case, to employ structures and algorithms designed for graphs. We say that a pure strongly connected $d$-dimensional simplicial complex is a \emph{pseudomanifold} if every $(d-1)$-face is contained in exactly two facets.
\begin{definition}
	For a pure $d$-dimensional pseudomanifold $\Delta$ the \emph{dual graph} $G(\Delta)$ is the graph on the vertex set $\left\lbrace F \in\Delta: \dim(F)=d\right\rbrace $ and with edge set $\left\lbrace \left\lbrace F_i,F_j\right\rbrace: \dim(F_i\cap F_j)=d-1  \right\rbrace$. 
\end{definition} 
\noindent
Given a $d$-dimensional pseudomanifold $\Delta$ and a pure subcomplex $\Phi\subseteq\partial\mathcal{C}_{d+1}$ that is a ball, we list all subgraphs of $G(\Delta)$ that are isomorphic to $G(\Phi)$ using an algorithm such as the VF2 algorithm \cite{CoFoSaVe}, from which we only keep those that correspond to induced subcomplexes. This check can be performed rather efficiently by direct inspection on the faces of $\Delta$.
\begin{remark}
	In general, the information encoded in the dual graph is not sufficient to reconstruct the simplicial complex. Let $\Phi$ be a $d$-ball. If $\Delta$ is a combinatorial $d$-manifold and $\Gamma\subseteq\Delta$ a subcomplex, then $G(\Gamma)\cong G(\Phi)$ does not imply $\Gamma\cong\Phi$. However, if $\Delta$ is in addition balanced and $\Phi\subseteq\partial\mathcal{C}_{d+1}$ then the following holds. 
\end{remark}
\begin{lemma}
	Let $\Delta$ be a balanced combinatorial $d$-manifold, $\Gamma\subseteq\Delta$ a pure subcomplex and $\Phi\subseteq\partial\mathcal{C}_{d+1}$ a $d$-ball. If $G(\Gamma)\cong G(\Phi)$, then $\Gamma\cong\Phi$.
\end{lemma}
\begin{proof}
	Assume that $\Gamma\ncong\Phi$. Since $G(\Gamma)\cong G(\Phi)$, $\Gamma$ is isomorphic to $\Phi$ modulo identifying vertices, without affecting $d$- and $(d-1)$-dimensional faces. Since $\Delta$ is balanced, the identification has to preserve the coloring, that is, a vertex $v$ has to be identified with a vertex $w$ of the same color. Finally we observe that if $v$ and $w$ are vertices of the same color in $\Phi\subseteq\partial\mathcal{C}_{d+1}$, then there exists a $(d-1)$-dimensional face $F\in \Phi$ such that $F\cup\{v\}\in \Phi$ and $F\cup\{w\}\in \Phi$. For an explicit description of all possible subcomplexes $\Phi$ we refer to \cite[Section 4.3]{2018arXiv180406270J}. The simplicial map identifying $v$ and $w$ then also identifies $F\cup\{v\}$ and $F\cup\{w\}$, which is in contradiction with $G(\Gamma)\cong G(\Phi)$.
\end{proof}
\noindent
Moreover once a flip $\Delta\longmapsto\chi_{\Phi}(\Delta)=:\Delta'$ is performed we do not need to rerun the check on the entire complex to list all the flippable subcomplexes of $\Delta'$, but it suffices to update the list locally, by considering only the induced subcomplexes of $\Delta'$ that are not induced subcomplexes of $\Delta$. Even though this idea allows to deal with relatively large $3$-dimensional complexes, higher dimensions appear to be still out of reach.\\
For the second problem, namely to decide which subcomplex to flip, we propose and combine two very naive strategies:
given a balanced pseudomanifold $\Delta$ we choose a subcomplex $\Phi$ among those which
\begin{itemize}
	\item maximize $\displaystyle\left| \left\lbrace v\in\chi_{\Phi}(\Delta): f_0(\lk_{\chi_{\Phi}(\Delta)}(v))=2d\right\rbrace \right|$,
	\item maximize $\displaystyle\sum_{v\in\chi_{\Phi}(\Delta),\dim(v)=0}(f_0(\lk_{\chi_{\Phi}(\Delta)}(v)))^2$.  
\end{itemize} 
We maximize the two functions in the order as presented above. With the first condition we simply maximize the number of potentially removable vertices, while maximizing the sum of squares of the vertex degrees forces the new triangulation to have an inhomogeneous degree distribution, and hence some very poorly connected vertices. 
\begin{remark}\label{rem: many irreducible}
	Typically, starting from a large triangulation, we cannot hope to reduce drastically the number of vertices through a sequence consisting only of down-flips, because irreducible triangulations are quite frequent. Even a restricted example like \Cref{big_graph} reveals several irreducible triangulations of $S^3$ on few vertices. We can overcome this inconvenience by interposing a certain number of random up-flips to avoid local minima.
\end{remark}
\begin{remark}
	We make no claim of efficiency, and we do not take into account the time needed to reduce a triangulation. Undoubtedly many details in the implementation can be improved, and the strategies refined. So far we were able to obtain small balanced $f$-vectors of all the $3$-dimensional examples considered.
	In general, the problem of finding subcomplexes of the boundary of the cross-polytope in a triangulation is significantly harder than finding subcomplexes of the boundary of the simplex. This fact, and the issue pointed out in \Cref{rem: many irreducible} make our code slower than the program BISTELLAR. In particular, the technique of simulated annealing employed in \cite{BjLu} does not appear to be effective in this setting.  
\end{remark}
\section{Surfaces and the Dunce Hat}
\noindent
The first complexes we consider are triangulations of compact $2$-manifolds. In this case the number of vertices uniquely determines the remaining entries of the $f$-vector. In \Cref{tab:table3} we display a list of minimal known $f$-vectors of several surfaces, as well as the $f$-vectors of the corresponding barycentric subdivisions, which is always the starting input for our procedure. Finally in the fourth column we report the smallest known $f$-vectors of balanced triangulations found via the program.
\subsection{Real projective plane}
We found a unique vertex minimal balanced triangulation $\Delta^{\mathbb{R}\mathbf{P}^{2}}_{9}$ of the real projective plane, which is depicted in two ways in \Cref{rp}. The $f$-vector is $f(\Delta^{\mathbb{R}\mathbf{P}^{2}}_{9})=(1,9,24,16)$. The non balanced minimal one has 6 vertices.
\begin{figure}[h]
	\centering
	\includegraphics{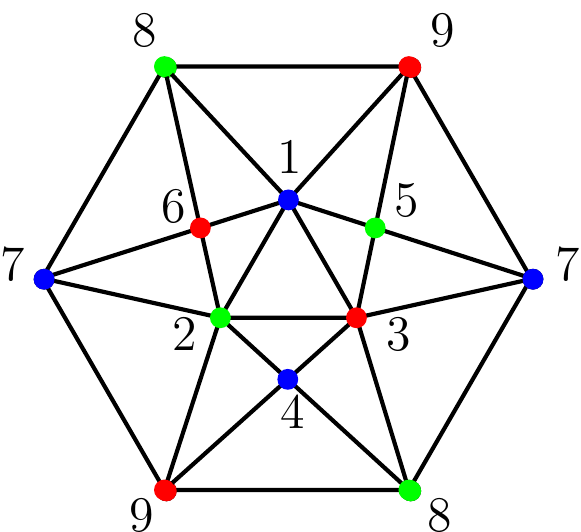}
	\quad\quad\quad\quad\quad
	\includegraphics{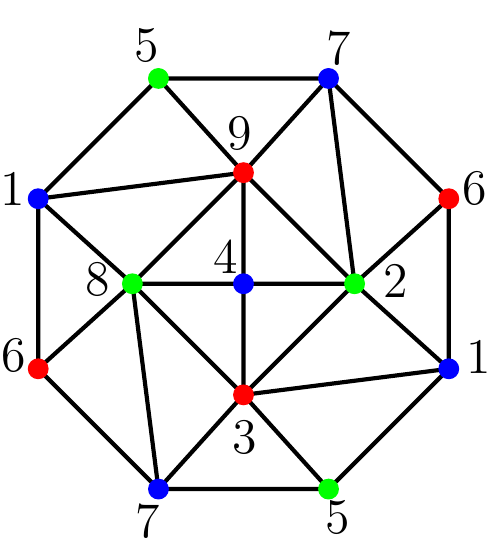}
	\caption{The simplicial complex $\Delta^{\mathbb{R}\mathbf{P}^{2}}_{9}$ represented as the quotient of a disk in two different ways.}
	\label{rp}
\end{figure} 
\begin{lemma}\label{rp2}
	The simplicial complex $\Delta^{\mathbb{R}\mathbf{P}^{2}}_{9}$ is a vertex (hence every $f_i$) minimal balanced triangulation of the projective plane.
\end{lemma}
\begin{proof}
	The claim follows from a result of Klee and Novik ([\cite{KN}, Proposition 6.1]) which states that any balanced triangulation of an homology $d$-manifold $\Delta$ that is not an homology $d$-sphere has at least three vertices in each color class. We briefly recall their argument. Assume there exists a color (say 1) whose class contains less than three vertices: if there is only one vertex colored with 1 then $\Delta$ is a cone, since every facet contains one vertex per color, and hence contractible. In the same way, since every $(d-1)$-face colored with $[d]\setminus \left\lbrace 1\right\rbrace$ is connected to exactly 2 vertices of color 1, if $\Delta$ has precisely two of color 1 is 2, say $v_1$ and $v_2$, then $\Delta$ is a suspension over the homology $(d-1)$-sphere $\lk_{\Delta}(v_1)\cong \lk_{\Delta}(v_2)$. Hence $\Delta$ is an homology sphere. 
\end{proof}
\begin{remark}
	Since triangulated surfaces on 9 vertices are listed in \cite{Lu}, we can check that $\Delta^{\mathbb{R}\mathbf{P}^{2}}_{9}$ is indeed the unique balanced triangulation of $\mathbb{R}\mathbf{P}^{2}$ on 9 vertices.  
\end{remark}
\begin{table}[h!]

	\[\begin{array}{l|l|l|l|l}
	\left| \Delta\right|  & \text{Min }f(\Delta) &  f(\text{Bd}(\Delta)) & \text{Min. Bal. $f$ known} & \text{Notes} \\\hline
	S^2 & (1,4,6,4) &  (1,14,36,24) & (1,6,12,8)^* & \partial\mathcal{C}_{3}\\
	\mathbb{T} & (1,7,21,14) & (1,42,126,84) & (1,9,27,18)^* & \text{see \cite{KN}} \\
	\mathbb{T}^{\#2} & (1,10,36,24) & (1,70,216,144) & (1,12,42,28) &\\
	\mathbb{T}^{\#3} & (1,10,42,28) & (1,80,252,168) & (1,14,54,36)& \\
	\mathbb{T}^{\#4} & (1,11,51,34) & (1,96,306,204)  & (1,14,60,36)& \\
	\mathbb{T}^{\#5} & (1,12,60,40) & (1,112,360,240)  & (1,16,72,48)& \\
	\mathbb{R}\mathbf{P}^{2} & (1,6,15,10) & (1,31,90,60) & (1,9,24,16)^* & \Delta^{\mathbb{R}\mathbf{P}^{2}}_{9}\\
	(\mathbb{R}\mathbf{P}^{2})^{\#2} & (1,8,24,16) & (1,48,144,96) & (1,11,33,22) & \\
	(\mathbb{R}\mathbf{P}^{2})^{\#3} & (1,9,30,20) & (1,59,180,120) & (1,12,39,26) & \\
	(\mathbb{R}\mathbf{P}^{2})^{\#4} & (1,9,33,22) & (1,64,198,132) & (1,12,42,28) & \\
	(\mathbb{R}\mathbf{P}^{2})^{\#5} & (1,9,36,24) & (1,69,216,144) & (1,13,48,32) & \\
	\end{array}
	\]
	\caption{A table reporting some small $f$-vectors of balanced surfaces. The symbol "*" indicates that the corresponding value of $f_0$ is the minimum number of vertices of a balanced triangulation of $|\Delta|$.}
	\label{tab:table3}
\end{table}
\subsection{Dunce hat}\label{section_dunce}
The dunce hat is a topological space which exhibits interesting properties: it is contractible but non-collapsible, and its triangulations are Cohen-Macaulay over any field but not shellable. For the definition of shellability we refer to \Cref{section_decomposition}, while we avoid defining Cohen-Macaulay simplicial complexes here, since they go beyond the aim of this article. The interested reader can find an extensive treatment of this topic in \cite{Stanley-greenBook}. The dunce hat can be visualized as a triangular disk whose edges are identified via a non-coherent orientation. Surprisingly, even if it is possible to construct a balanced triangulation by allowing only 3 vertices to be on the boundary of such disk, the vertex minimal one, which is depicted in \Cref{dunce}, is achieved when the singularity contains 4 vertices. Its $f$-vector is $f(\Delta^{\text{DH}})=(1,11,34,24)$. In the rest of this section we prove that this is indeed the least number of vertices that a balanced triangulation of the dunce hat can have. We call the \emph{singularity} of a triangulation of the dunce hat the $1$-dimensional subcomplex of faces whose link is not a sphere. Note that the link of edges in the singularity consists of three isolated vertices. The set of faces whose link is a sphere is called the \emph{interior} of the dunce hat, and it coincides with the interior of the triangular disk.\\
Since the dunce hat is not a manifold the number of vertices of a triangulation does not uniquely determine the other face numbers, but the number of vertices involved in the singularity also plays a role. If we let $f_0^{\text{sing}}$ be this number, then the $f$-vector $(1,f_0,f_1,f_2)$ of any triangulation of the dunce hat satisfies the following equations:
\begin{equation}\label{f_relations}
	\begin{cases}
	f_0-f_1+f_2=1\\
	f_0^{\text{sing}}+2f_1-3f_2=0
	\end{cases}.
\end{equation}
In particular it holds that $f_1=f_0^{\text{sing}}+3f_0-3$. We proceed now with a sequence of lemmas leading to \Cref{prop_dunce}, proving that the triangulation in \Cref{dunce} is indeed a balanced vertex minimal triangulation of the dunce hat.
\begin{figure}[h]
	\centering
	\includegraphics[scale=0.8]{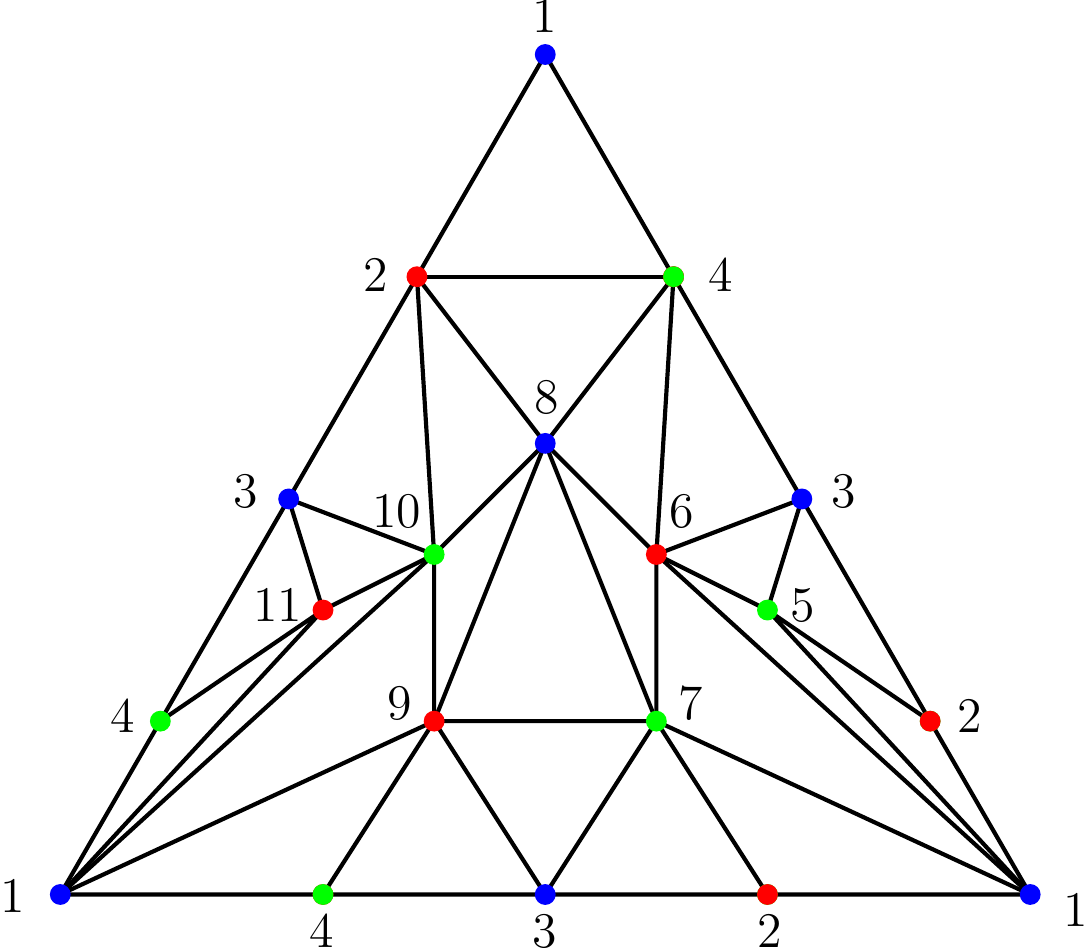}
	\caption{Minimal balanced triangulation $\Delta^{\text{DH}}$ of the dunce hat.}
	\label{dunce}
\end{figure} 

\begin{lemma}\label{at_least_2_per_class}
	Let $\Delta$ be a balanced $2$-dimensional Cohen-Macaulay complex that is not shellable. Then each color class contains at least two vertices.
\end{lemma}
\begin{proof}
	Assume there exists a color class containing only one vertex $v$.
	As discussed in \Cref{rp2} it follows that $\Delta$ is a cone over the $1$-dimensional subcomplex $\lk_{\Delta}(v)$. It is well known that the links of a Cohen-Macaulay simplicial complex are Cohen-Macaulay (see e.g., \cite{Stanley-greenBook}). But since every $1$-dimensional Cohen-Macaulay complex is shellable, and coning preserves shellability this implies that $\Delta$ is shellable.
\end{proof}
\begin{lemma}\label{at_least_3_per_class}
	Let $\Delta$ be a balanced $2$-dimensional Cohen-Macaulay complex that is not shellable. Assume moreover that every edge of $\Delta$ is contained in at least two triangles. Then each color class contains at least three vertices.
\end{lemma}
\begin{proof}
	By \Cref{at_least_2_per_class} we can assume that there are only two vertices of color 1, say $v_1$ and $v_2$. Let $\Delta_{\left[23\right]}$ be the subcomplex generated by all faces of $\Delta$ not containing color $1$. Since we assumed that every edge of $\Delta$ is contained in at least two triangles it follows that every edge $e$ of $\Delta_{\left[23\right]}$ is contained in the two triangles $e\cup\{v_1\}$ and $e\cup\{v_2\}$. This implies that $\Delta$ is obtained taking the suspension of $\Delta_{\left[23\right]}$, and hence  $\Delta_{\left[23\right]}=\lk_{\Delta}(v_1)=\lk_{\Delta}(v_2)$.
	In particular $\Delta_{\left[23\right]}$ is Cohen-Macaulay and $1$-dimensional, hence shellable. We deduce that $\Delta$ is the suspension over the shellable complex $\Delta_{\left[23\right]}$ and hence shellable, since suspension preserves shellability.   
\end{proof}
\begin{lemma}\label{f0sing=3}
	If $f_0^{\text{sing}}(\Delta)=3$ then $f_0(\Delta)\geq 10$.
\end{lemma}
	\begin{proof}
		 Observe that by \Cref{at_least_3_per_class} we need at least 9 vertices to triangulate the dunce hat in a balanced way, and the only possible configuration is $(n_1,n_2,n_3)=(3,3,3)$, where $n_i$ is the number of vertices of color $i$. Moreover note that in this case the singularity consists of one vertex per color class. Let us consider the edge $e$ in the singularity containing the colors $1$ and $2$. Since there are three copies of $e$ in the boundary of the disk, each of which needs to be completed to a triangle using an interior vertex of color $3$, we infer that $n_3\geq 4$, because another vertex of color $3$ already lies in the singularity.
	\end{proof}
	\begin{remark}
		We suspect that the bound in \Cref{f0sing=3} is far from being tight. Using our computer program, the smallest balanced triangulation obtained for the case $f_0^{\text{sing}}(\Delta)=3$ has $14$ vertices. However \Cref{f0sing=3} combined with \Cref{no10} suffices for our purpose.
	\end{remark}
	\noindent
	By \emph{bichromatic missing edge} we mean a pair of vertices ${i,j}$, such that $\left\lbrace i,j\right\rbrace \notin \Delta$, and $\kappa(i)\neq\kappa(j)$.
\begin{lemma}\label{at_least_10}
	Let $\Delta$ be a balanced triangulation of the dunce hat. Then $f_0(\Delta)\geq 10$.  Let $m$ be the number of bichromatic missing edges. If $f_0^{\text{sing}}(\Delta)+m\geq 7$ then $f_0(\Delta)\geq 11$.
\end{lemma}
\begin{proof}
	For any balanced $2$-dimensional simplicial complex with $n_i$ vertices of color $i$, for $i=1,2,3$, the number of edges is clearly bounded by above by the number of edges of the complete $3$-partite graph $K_{n_1,n_2,n_3}$, which equals $\left| E(K_{n_1,n_2,n_3})\right|=n_1n_2+n_1n_3+n_2n_3$. Combined with \Cref{f0sing=3}, which allows us to assume $f_0^{\text{sing}}(\Delta)\geq 4$, this yields 
	\begin{equation}\label{inequality_dunce}
	f_1(\Delta)=f_0^{\text{sing}}(\Delta)+3f_0(\Delta)-3\leq n_1n_2+n_1n_3+n_2n_3\leq \frac{f_0(\Delta)^2}{3},
	\end{equation}
	where the last inequality follows by maximizing the function $n_1n_2+n_1n_3+n_2n_3$, under the constraint $\sum_{i=1}^{3}n_i=f_0(\Delta)$. Solving the inequality $f_0^{\text{sing}}(\Delta)+3f_0(\Delta)-3\leq \frac{f_0(\Delta)^2}{3}$ for $f_0(\Delta)$ we obtain
	$$f_0(\Delta)\geq\dfrac{9+\sqrt{81+12f_0^{\text{sing}}(\Delta)-36}}{2}.$$
	By \Cref{f0sing=3} we can assume $f_0^{\text{sing}}(\Delta)\geq 4$, from which it follows that $f_0(\Delta)\geq 9,32$.
	The second statement follows in the same way by imposing $f_1(\Delta)\leq n_1n_2+n_1n_3+n_2n_3-m$ in the first inequality, which yields 
	$$f_0(\Delta)\geq\dfrac{9+\sqrt{81+12(f_0^{\text{sing}}(\Delta)+m)-36}}{2}.$$
	Under the assumption $f_0^{\text{sing}}(\Delta)+m\geq 7$ we obtain $f_0(\Delta)\geq 10,18$.
\end{proof}
\noindent
In order to show that the minimum number of vertices for a balanced triangulation of the dunce hat is 11 it remains to show that no such simplicial complex exists with $f_0(\Delta)=10$ and $f_0^{\text{sing}}(\Delta)\in\left\lbrace 3,4,5,6 \right\rbrace$. 
\begin{lemma}\label{no10}
	No balanced triangulation of the dunce hat on 10 vertices exists.
\end{lemma}
\begin{proof}
	Since any triangulation $\Delta$ of the dunce hat is Cohen-Macaulay, non-shellable, and has the property that every edge is contained in two or three triangles, \Cref{at_least_2_per_class} and \Cref{at_least_3_per_class} imply that every color class of $\Delta$ contains at least three vertices. Assume a balanced triangulation $\Delta$ on 10 vertices exists. There is a unique way to partition 10 vertices in three classes, such that every class contains at least three, namely $(n_1,n_2,n_3)=(3,3,4)$, where $n_i$ is the number of vertices of color $i$.
    Moreover, due to \Cref{at_least_10}, we can assume that $f_0^{\text{sing}}(\Delta)\in\left\lbrace 3,4,5,6 \right\rbrace$. In what follows we denote with $n_i^{\text{sing}}$ the number of vertices in the singularity of color $i$. Note the following facts:
    \begin{itemize}
    	\item[{\sf Claim 1:}] If $f_0^{\text{sing}}(\Delta)=3$ then there are at least four missing bichromatic edges. 
    	\begin{itemize}
    		\item Since the singularity is a triangle it must be colored using all three colors. This implies that $n_1^{\text{sing}}+n_2^{\text{sing}}\leq 2$, and hence there are at least four interior vertices of color $1$ or $2$. Denote with $v$ one of these vertices and assume $\kappa(v)=1$. Since the link of $v$ is a polygon with an even number of vertices, but there are 7 remaining vertices of color $2$ and $3$ ($3$ and $4$ respectively), then there exists a vertex $a$ of color $3$ such that $\left\lbrace v,a\right\rbrace \notin\Delta$. We obtain in this way a bichromatic missing edge for each of the four interior vertices of color $1$ and $2$.
    	\end{itemize}
    	\item[{\sf Claim 2}] If $f_0^{\text{sing}}(\Delta)=4$ then there are at least three missing bichromatic edges. 
    	\begin{itemize}
    		\item If $(n_1^{\text{sing}},n_2^{\text{sing}})\neq(2,2)$ then there are at least three vertices of color $1$ and $2$ in the interior of the disk and the link of these vertices is a polygon with an even number of vertices. Let $v$ be one of these vertices, and assume w.l.o.g. that $\kappa(v)=1$, where $\kappa$ is the coloring map of $\Delta$. Since $n_2+n_3(=n_1+n_3)=7$ only three of the vertices colored with $3$ can appear in the link of $v$. Hence there is at least one missing bichromatic edge for each of the three vertices.
    		\item If $(n_1^{\text{sing}},n_2^{\text{sing}})=(2,2)$ then there are exactly two vertices of color $1$ and $2$ (say $v$ and $w$) in the interior of $\Delta$ and their link is an even polygon. Again since $n_2+n_3=n_1+n_3=7$ each of these two vertices avoids at least one vertex of color 3 and there is at least one missing bichromatic edge for each of the two vertices. Let us denote with $\left\lbrace v,a \right\rbrace$ and $\left\lbrace w,b \right\rbrace$ these missing edges. If $a\neq b$ then since $a$ is in the interior and since the link of $a$ contains at most two vertices of color $1$, it can only contain two vertices of color $2$. Hence there is at least a third bichromatic missing edge $\left\lbrace z,a\right\rbrace$. If $a=b$ then the link of $a$ is the whole singularity (a square) and, in particular, there exists an edge $\left\lbrace x,y \right\rbrace$ in the interior of the disk whose endpoints are in the singularity. This yields a contradiction, because in the case considered two vertices in the singularity are either connected by an edge in the boundary of the disk, or they have the same color. 
    	\end{itemize}
    	\item[{\sf Claim 3}] If $f_0^{\text{sing}}(\Delta)=5$ then there are at least two missing bichromatic edges. 
    	\begin{itemize}
    		\item Since the singularity is a $5$-gon it must be colored using all the three color classes. This implies that $n_1^{\text{sing}}+n_2^{\text{sing}}\leq 4$, and hence there are at least two interior vertices of color $1$ or $2$. As in the previous paragraph each of these vertices must avoid at least one vertex of color $3$, giving rise to two bichromatic missing edges. 
    \end{itemize}
    \item[{\sf Claim 4}] If $f_0^{\text{sing}}(\Delta)=6$ then there is at least one missing bichromatic edge. 
    \begin{itemize}
    	\item If $(n_1^{\text{sing}},n_2^{\text{sing}})=(3,3)$ then the link of every interior vertex is the whole singularity, hence $\Delta$ is the join of a triangulation of $S^1$ with $4$ isolated vertices. This is clearly a contradiction.
    	\item If $(n_1^{\text{sing}},n_2^{\text{sing}})\neq(3,3)$ then there is at least one interior vertex of color $1$ or $2$. Once more its link cannot contain all the $7$ remaining vertices of a different color, so it must miss at least one vertex from the color class $3$. This produces a bichromatic missing edge.
    \end{itemize}
    \end{itemize}
    If we let $m$ be the number of bichromatic missing edges, then the four claims above imply $f_0^{\text{sing}}(\Delta)+m\geq 7$ for any $f_0^{\text{sing}}(\Delta)$. We conclude using \Cref{at_least_10}.
\end{proof}

\begin{proposition}\label{prop_dunce}
	The simplicial complex in \Cref{dunce} is a vertex minimal balanced triangulation of the dunce hat.
\end{proposition}
\begin{proof}
	The claim follows combining \Cref{f0sing=3}, \Cref{at_least_10} and \Cref{no10}, which show that no such triangulation exists on less than 11 vertices, for any value of $f_0^{\text{sing}}(\Delta)$. 
\end{proof}
\section{3-manifolds}\label{dimension_3}
\noindent
In this section we report some interesting and small balanced triangulations of $3$-manifolds found using our computer program.
\subsection{Real projective space}
We present a peculiar balanced triangulation $\Delta^{\mathbb{R}\mathbf{P}^{3}}_{16}$ of the real projective space on 16 vertices. An interesting feature of this complex is its strong symmetry: it is \emph{centrally symmetric} (i.e., there is a free involution acting) and all the vertex links are isomorphic to the $2$-sphere in \Cref{links}. Since the projective space is homeomorphic to the \emph{lens space} $L(2,1)$, a particular case of the following result of Zheng shows that our triangulation is vertex minimal.
\begin{proposition}\cite[Proposition 4.3]{Zh17}\label{lens_at_least_16}
	Any balanced triangulation of the lens space $L(p,q)$, with $p>1$, has at least 16 vertices.
\end{proposition}
\begin{table}[h]
	\resizebox{\textwidth}{!}{
\begin{tabular}{l l l l l l l l}
	\hline\hline
	$[0,2,3,7]$	&	$[0,4,6,14]$	&	$[1,2,5,14]$	&	$[1,5,10,15]$	&	$[2,3,7,11]$	&	$[2,11,12,13]$	&	$[3,7,9,11]$	&	$[6,8,9,15]$\\
	$[0,2,3,12]$	&	$[0,4,6,15]$	&	$[1,2,5,15]$	&	$[1,6,10,13]$	&	$[2,3,11,12]$	&	$[3,4,5,8]$	&	$[3,10,11,12]$	&	$[6,8,10,15]$\\
	$[0,2,7,15]$	&	$[0,4,7,15]$	&	$[1,2,7,14]$	&	$[1,6,10,15]$	&	$[2,5,8,13]$	&	$[3,4,7,8]$	&	$[4,5,8,13]$	&	$[6,9,11,14]$\\
	$[0,2,12,15]$	&	$[0,5,10,14]$	&	$[1,2,7,15]$	&	$[1,7,9,13]$	&	$[2,5,8,15]$	&	$[3,5,8,10]$	&	$[4,5,11,13]$	&	$[6,10,11,13]$\\
	$[0,3,4,5]$	&	$[0,6,9,14]$	&	$[1,4,6,13]$	&	$[1,7,9,14]$	&	$[2,5,11,13]$	&	$[3,6,8,9]$	&	$[4,5,11,14]$	&	$[7,8,9,13]$\\
	$[0,3,4,7]$	&	$[0,6,9,15]$	&	$[1,4,6,15]$	&	$[1,9,12,13]$	&	$[2,5,11,14]$	&	$[3,6,8,10]$	&	$[4,6,11,13]$	&	$[7,9,11,14]$\\
	$[0,3,5,10]$	&	$[0,9,12,14]$	&	$[1,4,7,13]$	&	$[1,9,12,14]$	&	$[2,7,11,14]$	&	$[3,6,9,11]$	&	$[4,6,11,14]$	&	$[8,9,12,13]$\\
	$[0,3,10,12]$	&	$[0,9,12,15]$	&	$[1,4,7,15]$	&	$[1,10,12,13]$	&	$[2,8,12,13]$	&	$[3,6,10,11]$	&	$[4,7,8,13]$	&	$[8,9,12,15]$\\
	$[0,4,5,14]$	&	$[0,10,12,14]$	&	$[1,5,10,14]$	&	$[1,10,12,14]$	&	$[2,8,12,15]$	&	$[3,7,8,9]$	&	$[5,8,10,15]$	&	$[10,11,12,13]$\\
	\hline\hline
\end{tabular}}
	\caption{The list of facets of $\Delta^{\mathbb{R}\mathbf{P}^{3}}_{16}$.}
\end{table}


\begin{figure}[h]
	\centering
	\includegraphics[scale=0.8]{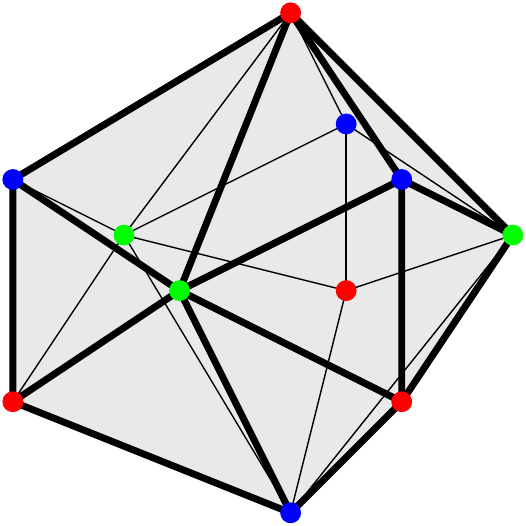}
	\caption{The (all isomorphic) vertex links of the triangulation $\Delta^{\mathbb{R}\mathbf{P}^{3}}_{16}$.}
	\label{links}
\end{figure} 
	
\subsection{Small balanced triangulations of $3$-manifolds}\label{dimension_3_sub}
In \Cref{tab:table2} we report the smallest known $f$-vectors of balanced triangulations of several $3$-manifolds. We point out that some of these triangulations were previously known, and they are referenced through this section. For instance Klee and Novik \cite{KN} proved the existence of a $d$-dimensional simplicial complex on $3d+3$ vertices which provides a vertex minimal balanced triangulation of $S^{d-1}\times S^1$ when $d$ is even, and of the twisted bundle $S^{d-1}\dtimes S^1$ when $d$ is odd. Moreover they construct balanced triangulations of both $S^{d-1}\times S^1$ and $S^{d-1}\dtimes S^1$ on $3d+5$ vertices. These constructions, combined with a result of Zheng (\cite{Zh16}), show that the second and third line in \Cref{tab:table2} correspond to vertex minimal balanced triangulations. As previously discussed, via \Cref{lens_at_least_16} we can conclude that also $\Delta^{\mathbb{R}\mathbf{P}^{3}}_{16}$ and a triangulation of the lens space $L(3,1)$ constructed in \cite{Zh17} are balanced vertex minimal. Most notably the mentioned triangulations of $S^{d-1}\times S^1$ ($d$ even), $S^{d-1}\dtimes S^1$ ($d$ odd) and of $L(3,1)$ are \emph{balanced neighborly}, that is they do not have bichromatic missing edges.\\
In the rest of the table we report the minimal balanced $f$-vectors achieved for different $3$-manifolds, such as several lens spaces $L(p,q)$, connected sums, two additional spherical 3-manifolds called the octahedral space and the cube space, and the Poincar\'{e} homology $3$-sphere. For a more extensive treatment of this topic we refer to \cite{LuThesis}. A classical theorem in topology by Edwards and Cannon states that the $k$-fold suspension of any homology $d$-sphere is homeomorphic to $S^{d+k}$, even though it is not a combinatorial sphere. Since balancedness is preserved by taking suspension we obtain a family of non-combinatorial balanced triangulations of $S^{d}$, for $d\geq 5$. For $d=5$ we report the smallest $f$-vector known for a non-combinatorial balanced sphere.
\begin{corollary}
	There exists a balanced non-combinatorial $5$-sphere with $f$-vector $(1,30,288,1132,2106,1848,616)$. Moreover by taking further suspensions we obtain balanced non-combinatorial $d$-spheres on $2d+20$ vertices, for every $d\geq 5$.
\end{corollary}
\begin{remark}
	As it was pointed out in \cite{BjLu} there exists a procedure introduced by Datta to construct the suspension by increasing the number of vertices only by one. Unfortunately this one point suspension does not preserve balancedness.
\end{remark}
	\noindent
	The lists of facets of all the triangulations appearing in \Cref{tab:table2} can be found in \cite{GitRepo}.
\subsection{The connected sum of $S^2$ bundles over $S^1$ and the balanced Walkup class}\label{conjecture_section}
    The lower bound theorem for manifolds (actually true for normal pseudomanifolds, see \cite{MUR}) gives a bound for the number of edges of a triangulation $\Delta$ of an $\FF$-homology manifold with a certain number of vertices, depending of $\widetilde{\beta}_1(\Delta;\FF)$. It is an interesting refinement of the lower bound theorem for homology spheres, obtained from the study of algebraic invariants of Buchsbaum graded rings. Juhnke-Kubitzke, Murai, Novik and Sawaske proved a balanced analog of this bound (see \cite{Juhnke:Murai:Novik:Sawaske}), and established a conjecture of Klee and Novik \cite[Conjecture 4.14]{KN} for the characterization of the case of equality, when the dimension is greater or equal to $4$. Let $\Delta$ and $\Gamma$ be pure balanced simplicial complexes of the same dimension on disjoint vertex sets, let $F,G$ be two facets of $\Delta$ and $\Gamma$ respectively and let $\varphi:F\longrightarrow G$ be a bijection. Then the \emph{connected sum} $\Delta\#\Gamma$ is the simplicial complex obtained from $\Delta\setminus F$ and $\Gamma\setminus G$ identifying $v$ and $\varphi(v)$, for every $v\in F$. Let $\Delta$ be a balanced simplicial complex, $F,G$ two facets of $\Delta$ and $\varphi:F\longrightarrow G$ a bijection such that $\lk_{\Delta}(v)\cap\lk_{\Delta}(\varphi(v))=\left\lbrace \emptyset\right\rbrace $ and $\kappa(v)=\kappa(\varphi(v))$, for all $v\in F$. We say that the simplicial complex obtained from $\Delta$ removing $F$ and $G$ and identifying $v$ with $\varphi(v)$ is a \emph{balanced handle addition}. Note that this operations preserves balancedness, as well as the property of being an homology manifold. We define the \emph{balanced Walkup class} $\mathcal{BH}_d$ as the set of all balanced simplicial complexes obtained from $\partial\mathcal{C}_{d+1}$ by successively applying the operations of connected
    sums with $\partial\mathcal{C}_{d+1}$ and balanced handle additions. In particular the set of balanced spheres on $n$ vertices obtained via connected sums of $\frac{n}{d}-1$ copies of $\partial\mathcal{C}_{d+1}$, called \emph{cross-polytopal stacked spheres}, is a subset of the balanced Walkup class.   \begin{theorem}\cite{Juhnke:Murai:Novik:Sawaske}\label{GLBT}
	Let $\Delta$ be a connected $d$-dimensional balanced $\FF$-homology manifold, with $d\geq 3$. Then
	\begin{equation}\label{eqn_GLBT}
	2f_1(\Delta)-3df_0(\Delta)\geq 4\binom{d+1}{2}(\widetilde{\beta}_1(\Delta;\FF)-1).
	\end{equation}
	Moreover if $d\geq 4$ equality holds if and only if $\Delta$ is in the balanced Walkup class.
\end{theorem} 
    \noindent
	\Cref{GLBT} leaves unsolved the case of equality when $d=3$, which is still part of Conjecture 4.14 in \cite{KN}.
\begin{conjecture}[\cite{KN}]\label{KLconjecture}
	Let $\Delta$ be a connected $3$-dimensional balanced $\FF$-homology manifold. Then $2f_1(\Delta)-9f_0(\Delta)= 24(\widetilde{\beta}_1(\Delta;\FF)-1)$ if and only if $\Delta$ is in the balanced Walkup class. 
\end{conjecture}
	\noindent
	Using our computer program we found two balanced triangulation of $(S^2\times S^1)^{\#2}$ and $(S^2\dtimes S^1)^{\#2}$ respectively with $f$-vector $(1,16,84,136,68)$. Since $\widetilde{\beta}_1((S^2\times S^1)^{\#2};\FF)=\widetilde{\beta}_1((S^2\dtimes S^1)^{\#2};\FF)=2$, it is easy to see that both triangulations attain equality in \Cref{eqn_GLBT}. In light of \Cref{KLconjecture} it is natural to ask if these two simplicial complexes belong to the balanced Walkup class. We answer positively by giving an explicit decomposition. In what follows we denote with $\partial\mathcal{C}_{4}(v_1,v_2,v_3,v_4,w_1,w_2,w_3,w_4)$ the boundary of the cross-polytope on the vertex set $\left\lbrace v_1,v_2,v_3,v_4,w_1,w_2,w_3,w_4\right\rbrace$, such that $\left\lbrace v_i,w_i\right\rbrace$ is not an edge for any $i=1,\dots,4$.
	\begin{itemize}
		\item $(S^2\dtimes S^1)^{\#2}$:
		\begin{itemize}
			\item As it was pointed out in the first paragraph of \Cref{conjecture_section} there is a balanced simplicial complex $\Delta^{S^2\dtimes S^1}_{12}$ on $12$ vertices which triangulates $S^2\dtimes S^1$. It can be obtained from three copies of $\partial\mathcal{C}_{4}$, namely $\partial\mathcal{C}_{4}(x_1,\dots,x_4,y'_1,\dots,y'_4)$, $\partial\mathcal{C}_{4}(y_1,\dots,y_4,z'_1,\dots,z'_4)$ and $\partial\mathcal{C}_{4}(z_1,\dots,z_4,x'_1,\dots,x'_4)$, via two connected sums and one handle addition, identifying the vertices $v_i$ and $v'_i$ (see \cite{KN}). Clearly it belongs to  in $\mathcal{BH}_3$. 
			\item For any facet $F=\left\lbrace r_1,r_2,r_3,r_4\right\rbrace$ of $\Delta^{S^2\dtimes S^1}_{12}$ we can take the connected sum with $\partial\mathcal{C}_{4}(s_1,\dots,s_4,\linebreak r'_1,\dots,r'_4)$ and $\partial\mathcal{C}_{4}(t_1,\dots,t_4,s'_1,\dots,s'_4)$, again identifying $r_i$, with $r'_i$ and $s_i$ with $s'_i$.
			\item Finally we can choose any other facet $G=\left\lbrace u_1,u_2,u_3,u_4\right\rbrace$ of $\Delta^{S^2\dtimes S^1}_{12}$ such that $F\cap G=\emptyset$ and such that the distance from $F$ and $G$ (measured on the dual graph) is even, and perform balanced handle addition identifying the vertices $t_i$ with $u_i$.
			Since the vertices in the link of $t_i$ are a subset of $\left\lbrace s_1,s_2,s_3,s_4 \right\rbrace$ and $F\cap G=\emptyset$, we conclude that the links of $t_i$ and $u_i$ do not intersect, and hence the handle addition is well defined. 
			Note that the last choice can produce non-isomorphic triangulations of $(S^2\dtimes S^1)^{\#2}$, all of which sit inside $\mathcal{BH}_3$. 
		\end{itemize} 
	    \item $(S^2\times S^1)^{\#2}$:
	    \begin{itemize}
	    	\item With a similar construction as in the previous case we can triangulate the orientable bundle $S^2\times S^1$ with four copies of $\partial\mathcal{C}_{4}$, namely $\partial\mathcal{C}_{4}(x_1,\dots,x_4,y'_1,\dots,y'_4)$, $\partial\mathcal{C}_{4}(y_1,\dots,y_4,z_1,\dots,z_4)$, $\partial\mathcal{C}_{4}(w_1,\dots,w_4,x''_1,y''_2,y''_3,y''_4)$ and $\partial\mathcal{C}_{4}(w'_1,\dots,w'_4,y''_1,z_2,z_3,z_4)$. Here we perform three connected sums and a balanced handle addition identifying vertices $v_i$, $v'_i$ and $v''_i$. We denote this simplicial complex on $16$ vertices with $\Delta^{S^2\times S^1}_{16}$ (note that it is not the vertex minimal balanced triangulation of $S^2\times S^1$).
	    	\item We now pick the facets $F=\left\lbrace x_1,x_2,x_3,x_4 \right\rbrace$ and $G=\left\lbrace w_1,z_2,z_3,z_4\right\rbrace$ of $\Delta^{S^2\times S^1}_{16}$, and we observe that the link of any vertex in $F$ (respectively $G$) does not contain any vertex in $G$ (respectively $F$). Moreover $F$ and $G$ have an even distance (with respect to the dual graph of $\Delta^{S^2\times S^1}_{16}$).
	    	\item Finally we perform a connected sum and subsequent handle addition using $\partial\mathcal{C}_{4}(x'''_1,\dots,x'''_4,\linebreak w'''_1,z'''_2,z'''_3,z'''_4)$, where the identifications are between $v_i$ and $v'''_i$.
	    \end{itemize}
	\end{itemize}
\begin{remark}
	The description of the construction for $(S^2\dtimes S^1)^{\#2}$ mostly relies on the simplicial complex $\Delta^{S^2\dtimes S^1}_{12}$. Since in \cite{KN} the authors showed that the analog construction in dimension $d$ provides a triangulation on $3d$ vertices of $S^{d-1}\dtimes S^1$ if $d$ is even and of $S^{d-1}\times S^1$ if $d$ is odd, the extra handle that we add provides triangulations on $4d$ vertices of $(S^{d-1}\dtimes S^1)^{\#2}$ if $d$ is odd and of $(S^{d-1}\times S^1)^{\#2}$ if $d$ is even.
\end{remark}	

\begin{table}[h!]

	\[\begin{array}{l|l|l|l|l}
	\left| \Delta\right|  & \text{Min }f(\Delta) &  f(\text{Bd}(\Delta)) & \text{Min. Bal. $f$ known} & \text{Notes} \\\hline
	S^3 & (1,5,10,10,5) & (1,30, 150, 240, 120) & (1,8,24,32,16)^* & \partial\mathcal{C}_{4}\\
	S^2\times S^1 & (1,10,42,64,32) & (1,148,916,1536,768) & (1,14,64,100,50)^* & \text{see \cite{KN}}\\
	S^2\dtimes S^1 & (1,9,36,54,27) & (1,126,774,1296,648) & (1,12,54,84,42)^* & \text{see \cite{KN}}\\
	\mathbb{R}\mathbf{P}^{3} & (1,11,51,80,40) & (1, 182, 1142, 1920, 960) & (1, 16, 88, 144, 72)^* & \Delta^{\mathbb{R}\mathbf{P}^{3}}_{16}\\
	L(3,1) & (1,12,66,108,54) & (1, 240, 1536, 2592, 1296) & (1, 16, 96, 160, 80)^* & \text{see \cite{Zh17}}\\
	L(4,1) & (1,14,84,140,70) & (1, 308, 1988, 3360, 1680) & (1, 20, 132, 224, 112) &\\
	L(5,1) & (1,15,97,164,82) & (1, 358, 2326, 3936, 1968) & (1, 22, 152, 260, 130) &\\
	L(5,2) & (1,14,86,144,72) & (1, 316, 2044, 3456, 1728) & (1,20,132,224,112)&\\
	L(6,1) & (1,16,110,188,94) & (1,408,2664,4512,2256) & (1,24,176,304,152) & \\
	(S^2\times S^1)^{\#2} & (1,12,58,92,46) & (1, 208, 1312, 2208, 1104) & (1,16,84,136,68)& \text{see \Cref{conjecture_section}}\\
	(S^2\dtimes S^1)^{\#2} & (1,12,58,92,46) & (1, 208, 1312, 2208, 1104) & (1,16,84,136,68)& \text{see \Cref{conjecture_section}}\\
	(S^2\times S^1)\#\mathbb{R}\mathbf{P}^{3} & (1,14,73,118,59) & (1, 264, 1680, 2832, 1416) & (1,20,118,196,98) & \\
	(\mathbb{R}\mathbf{P}^{3})^{\#2} & (1,15,86,142,71) & (1,314,2018,3408,1704) & (1,21,137,232,116) & \\
	(S^2\times S^1)^{\#3} & (1,13,72,118,59) & (1,262,1678,2832,1416) & (1,20,118,196,98)& \\
	(S^2\dtimes S^1)^{\#3} & (1,13,72,118,59) & (1,262,1678,2832,1416) & (1,19,111,184,92)& \\
	S^1\times S^1\times S^1 & (1,15,105,180,90) & (1,390,2550,4320,2160) & (1,24,168,288,144)& \\
	\text{Oct. space} & (1,15,102,174,87) & (1,378,2466,4176,2088) & (1,24,168,288,144)& \\
	\text{Cube space} & (1,15,90,150,75) & (1,330,2130,3600,1800) & (1,23,157,268,134)& \\
	\text{Poincar\'{e}} & (1,16,106,180,90) & (1, 392, 2552, 4320, 2160) & (1, 26, 180, 308, 154) &\\
	\mathbb{R}\mathbf{P}^{2}\times S^1 & (1,14,84,140,70) & (1, 308, 1988, 3360, 1680) & (1, 24, 156, 264, 132) &\\
	\text{Triple-trefoil} & (1,18,143,250,125) & (1,536,3536,6000,3000) & (1, 28, 204, 352, 176) & \Delta^{3T}_{28}\\
	\text{Double-trefoil} & (1,16,108,184,92) & (1, 400, 2608, 4416, 2208) & (1, 22, 136, 228, 114) & \Delta^{2T}_{22}\\
	\end{array}
	\]
	\caption{A table reporting some small $f$-vectors of balanced $3$-manifolds.  The symbol "*" indicates that the corresponding value of $f_0$ is the minimum number of vertices of a balanced triangulation of $|\Delta|$.}
	\label{tab:table2}
\end{table}
\subsection{Non-vertex decomposable and non-shellable balanced $3$-spheres.}\label{section_decomposition}
In this paragraph we exhibit two interesting balanced triangulations of the $3$-sphere, namely one that is shellable but not vertex decomposable and a second one which is not constructible, and hence not shellable. We start with some definitions.
\begin{definition}
	A pure $d$-dimensional simplicial complex $\Delta$ is \emph{vertex decomposable} if it is the $d$-simplex or there exists a vertex $v$ such that $\lk_{\Delta}(v)$ and $\Delta\setminus v:=\left\lbrace F\in\Delta: v\notin F \right\rbrace $ are vertex decomposable. 
\end{definition}
\begin{definition}
	A pure $d$-dimensional simplicial complex is \emph{shellable} if there exists an ordering $F_1,\dots,F_m$ of its facets such that the complex $\left\langle F_1,\dots,F_{i-1}\right\rangle \cap\left\langle F_i\right\rangle$ is pure and $(d-1)$-dimensional for every $1\leq i \leq m$. Such an ordering is called \emph{shelling order}.
\end{definition}
\begin{definition}
	A pure $d$-dimensional simplicial complex $\Delta$ is \emph{constructible} if $\Delta\cong2^{[d+1]}$ or $\Delta=\Delta_1\cup\Delta_2$, where $\Delta_1$, $\Delta_2$ and $\Delta_1\cap\Delta_2$ are constructible. 
\end{definition}
\noindent
It is well known (see e.g. \cite{Bj95}) that these three families of simplicial complexes are related by the following hierarchy:
$$\left\lbrace \text{vertex decomposable}\right\rbrace \subseteq \left\lbrace \text{shellable} \right\rbrace \subseteq\left\lbrace \text{constructible} \right\rbrace.$$
In particular while there exist shellable $3$-spheres which are not vertex decomposable, the existence of constructible, but not shellable $3$-spheres is still open. In order to obtain interesting, possibly small balanced triangulations we again start from the barycentric subdivision of two distinct triangulations of the $3$-sphere with a sufficiently complicated knot embedded in their \emph{skeleton} (i.e., the subcomplex of all faces of dimension at most 1). For instance we turn our attention to the connected sum of $2$ or $3$ \emph{trefoil knots}, called a \emph{double-trefoil} and a \emph{triple-trefoil}. The reason for this choice is that in general the barycentric subdivision might turn non-shellable simplicial complexes into shellable ones, while complicated knots are obstructions to shellability even after the subdivision. We employ the following rephrasing of results by Ehrenborg and Hachimori (\cite{EhHa}), and Hachimori and Ziegler (\cite{HaZi}).
\begin{theorem}\label{knot_BZ}
	Let $\Delta$ be a triangulation of a $3$-sphere.
	\begin{itemize}
		\item (\cite{HaZi}) If the skeleton of $\Delta$ contains a double-trefoil knot on 6 edges then $\Delta$ is not vertex decomposable.
		\item (\cite{EhHa}) If the skeleton of $\Delta$ contains a triple-trefoil knot on 6 edges then $\Delta$ is not constructible (hence not shellable).
	\end{itemize}
\end{theorem}
\noindent
For an introduction to knot theory and a rigorous definition of complicatedness of knots we defer to a work of Benedetti and Lutz (\cite{BeLu}), where triangulations of the $3$-sphere containing the double and triple-trefoil knot on 3 edges were constructed: the first has 16 vertices (see $S_{16,92}$ in \cite{BeLu}), while the second has 18 vertices ($S_{18,125}$). Using our computer program we take the barycentric subdivision of these two complexes and we reduce them only applying cross-flips preserving the subdivision of the knots, which consist of 6 vertices. More precisely we only allow flips of the form $\Delta\longmapsto\chi_{\Phi}(\Delta)$, where the interior of $\Phi$ does not contain any of the 6 edges of the knot. \Cref{knot_BZ} guarantees that in this way the obstructions for vertex decomposability and shellability are preserved, which yields the following result.
\begin{proposition}
	There exist balanced triangulations $\Delta^{2T}_{22}$ and $\Delta^{3T}_{28}$ of the $3$-sphere that are:
	\begin{itemize}
		\item Shellable but not vertex decomposable ($\Delta^{2T}_{22}$), and $f(\Delta^{2T}_{22})=(1,22,136,228,114)$.
		\item Non constructible, hence not shellable ($\Delta^{3T}_{28}$),and $f(\Delta^{3T}_{28})=(1,28,204,352,176)$.
	\end{itemize}
\end{proposition} 
\noindent
In \Cref{facets_double_knot} and \Cref{facets_triple_knot} we report the list of facets of these two simplicial complexes.
\begin{table}[h] 
	\resizebox{\textwidth}{!}{
\begin{tabular}{l l l l l l l l}
	\hline\hline
	$[6, 19, v_3, v_{12}]$	&	 $[14, 19, 21, v_{23}]$	&	$[6, 16, v_3, v_{12}]$	&	$[14, 19, v_3, v_{12}]$	&	$[10, 11, 19, v_{2}]$	&	$[14, 19, v_{1}, v_{23}]$	&	$[7,9,14,v_1]$	&	$[7,8,13,v_{23}]$\\
	$[6, 19, v_3, v_{13}]$	&	 $[15, 19, v_{2}, v_{13}]$	&	 $[6, 16, v_3, v_{13}]$	&	 $[6, 9, 13, 17]$	&	 $[10, 15, 19, v_{2}]$	&	 $[14, 20, v_{1}, v_{23}]$	&	 $[7,9,14,v_3]$	&	 $[7,8,13,v_{13}]$\\
	$[6, 19, 21, v_{12}]$	&	 $[17, 18, 21, v_{12}]$	&	 $[11, 16, v_3, v_{13}]$	&	 $[10, 16, 18, v_{2}]$	&	 $[9, 11, 13, 17]$	&	 $[12, 15, v_{1}, v_{23}]$	&	 $[7,14,v_3,v_{12}]$	&	 $[12,14,v_2,v_{12}]$\\
	$[14, 19, 21, v_{12}]$	&	 $[16, 18, 21, v_{12}]$	&	 $[11, 16, v_3, v_{12}]$	&	 $[10, 17, 18, v_{2}]$	&	 $[9, 11, 17, v_{1}]$	&	 $[8, 12, v_{1}, v_{23}]$	&	 $[7,15,v_3,v_{12}]$	&	 $[12,15,v_2,v_{12}]$\\
	$[6, 17, 21, v_{12}]$	&	 $[6, 16, v_{1}, v_{12}]$	&	 $[11, 16, v_{2}, v_{13}]$	&	 $[17, 18, v_{2}, v_{23}]$	&	 $[9, 11, 12, 13]$	&	 $[8, 20, v_{1}, v_{23}]$	&	 $[7,15,v_3,v_{23}]$	&	 $[7,15,v_2,v_{12}]$\\
	$[6, 17, v_{1}, v_{12}]$	&	 $[6, 10, 16, v_{1}]$	&	 $[6, 16, v_{2}, v_{13}]$	&	 $[16, 18, v_{2}, v_{23}]$	&	 $[11, 12, 13, v_{12}]$	&	 $[8, 13, 20, v_{23}]$	&	 $[7,13,15,v_{23}]$	&	 $[7,14,v_2,v_{12}]$\\
	$[17, 18, v_{1}, v_{12}]$	&	 $[14, 16, 21, v_{12}]$	&	 $[6, 10, 16, v_{2}]$	&	 $[11, 16, v_{2}, v_{23}]$	&	 $[12, 13, 14, v_{12}]$	&	 $[8, 9, 12, v_3]$	&	 $[7,13,15,v_{13}]$	&	 $[7,14,v_2,v_{13}]$\\
	$[16, 18, v_{1}, v_{12}]$	&	 $[6, 19, 21, v_{23}]$	&	 $[6, 10, 12, v_{2}]$	&	 $[11, 17, v_{2}, v_{23}]$	&	 $[11, 12, v_3, v_{12}]$	&	 $[7, 8, 9, v_3]$	&	 $[7,15,v_2,v_{13}]$	&	 $[7,14,v_1,v_{13}]$\\
	$[6, 9, 17, v_{1}]$	&	 $[7, 8, 9, v_{1}]$	&	 $[6, 12, v_{2}, v_{13}]$	&	 $[11, 13, 17, v_{23}]$	&	 $[9, 11, 12, v_3]$	&	 $[8, 12, v_3, v_{23}]$	&	 $[12,13,14,v_{13}]$	&	 $[7,8,v_1,v_{13}]$\\
	$[6, 9, 12, v_{1}]$	&	 $[10, 16, 18, v_{1}]$	&	 $[6, 12, 13, v_{13}]$	&	 $[11, 13, 16, v_{23}]$	&	 $[9, 11, 19, v_3]$	&	 $[12, 15, v_3, v_{23}]$	&	 $[13,14,20,v_{13}]$	&	 \\
	$[6, 10, 12, v_{1}]$	&	 $[10, 17, 18, v_{1}]$	&	 $[6, 13, 19, v_{13}]$	&	 $[13, 14, 16, v_{23}]$	&	 $[9, 14, 19, v_3]$	&	 $[12, 15, v_3, v_{12}]$	&	 $[7,8,v_3,v_{23}]$	&	 \\
	$[8, 9, 12, v_{1}]$	&	 $[14, 16, 21, v_{23}]$	&	 $[13, 15, 19, v_{13}]$	&	 $[11, 13, 16, v_{12}]$	&	 $[10, 12, 15, v_{2}]$	&	 $[10, 11, 17, v_{1}]$	&	 $[8,13,20,v_{13}]$	&	 \\
	$[6, 9, 12, 13]$	&	 $[16, 18, 21, v_{23}]$	&	 $[6, 13, 19, v_{23}]$	&	 $[10, 11, 17, v_{2}]$	&	 $[10, 12, 15, v_{1}]$	&	 $[10, 11, 19, v_{1}]$	&	 $[14,20,v_1,v_{13}]$	&	 \\
	$[11, 19, v_3, v_{13}]$	&	 $[17, 18, 21, v_{23}]$	&	 $[13, 15, 19, v_{23}]$	&	 $[13, 14, 20, v_{23}]$	&	 $[10, 15, 19, v_{1}]$	&	 $[9, 11, 19, v_{1}]$	&	 $[8,20,v_1,v_{13}]$	&	 \\
	$[11, 19, v_{2}, v_{13}]$	&	 $[6, 17, 21, v_{23}]$	&	$[6, 13, 17, v_{23}]$	&	$[13, 14, 16, v_{12}]$	& $[15, 19, v_{1}, v_{23}]$		&	$[9, 14, 19, v_{1}]$	&	$[12,14,v_2,v_{13}]$	&\\
	\hline\hline
\end{tabular}}
 \caption{The list of facets of $\Delta^{2T}_{22}$. The order given by the columns (top to bottom and left to right) is a shelling order. The 6 vertices of the double-trefoil knot are labeled $v_F$, according to the face $F$ they correspond to in the barycentric subdivision.}
 \label{facets_double_knot}
\end{table}

\begin{table}[h] 
	\resizebox{\textwidth}{!}{
		\begin{tabular}{l l l l l l l l}
			\hline\hline
			$[0,2,6,v_{13}]$ & $[0,10,21,v_3]$ & $[5,11,12,v_{23}]$ & $[8,17,20,v_{12}]$ & $[0,8,17,v_{12}]$ & $[3,7,11,v_1]$ & $[0,10,21,v_1]$ & $[14,18,19,v_{13}]$\\
			$[1,11,19,v_1]$ & $[10,20,v_3,v_{13}]$ & $[1,14,18,19]$ & $[5,16,v_2,v_{13}]$ & $[11,19,21,v_1]$ & $[2,4,v_2,v_{13}]$ & $[5,12,16,v_{13}]$ & $[5,6,14,v_{12}]$\\
			$[8,16,19,21]$ & $[2,4,v_1,v_{13}]$ & $[5,17,18,v_{23}]$ & $[0,17,v_2,v_{13}]$ & $[3,7,16,v_1]$ & $[11,13,v_1,v_{23}]$ & $[0,1,8,17]$ & $[0,11,21,v_3]$\\
			$[6,14,19,v_{13}]$ & $[1,8,17,20]$ & $[1,10,18,20]$ & $[4,16,v_1,v_{23}]$ & $[2,5,6,v_{23}]$ & $[5,11,18,v_{23}]$ & $[5,14,v_1,v_{12}]$ & $[11,12,13,v_{23}]$\\
			$[1,13,17,v_2]$ & $[6,17,19,v_{13}]$ & $[5,6,10,v_{23}]$ & $[13,17,21,v_2]$ & $[5,17,v_3,v_{23}]$ & $[1,14,19,v_1]$ & $[0,16,v_1,v_{23}]$ & $[2,6,20,v_{13}]$\\
			$[14,19,21,v_1]$ & $[10,18,20,v_{13}]$ & $[0,15,17,v_3]$ & $[7,15,16,v_3]$ & $[6,14,19,21]$ & $[0,6,10,15]$ & $[2,18,19,v_{12}]$ & $[1,10,18,19]$\\
			$[3,7,11,v_3]$ & $[17,19,v_3,v_{23}]$ & $[5,14,21,v_1]$ & $[6,14,20,v_{13}]$ & $[5,6,16,21]$ & $[9,15,16,v_1]$ & $[10,20,21,v_2]$ & $[5,16,21,v_2]$\\
			$[0,1,8,16]$ & $[0,2,v_2,v_{13}]$ & $[9,12,15,16]$ & $[7,11,18,v_{13}]$ & $[1,14,20,v_1]$ & $[10,12,13,v_{12}]$ & $[4,16,v_1,v_{13}]$ & $[5,10,12,v_{23}]$\\
			$[8,11,20,21]$ & $[17,18,19,v_{13}]$ & $[9,16,v_1,v_{13}]$ & $[2,12,20,v_{12}]$ & $[10,18,19,v_{12}]$ & $[0,6,10,v_{23}]$ & $[2,4,v_2,v_{23}]$ & $[6,14,20,v_{12}]$\\
			$[11,20,v_3,v_{13}]$ & $[5,6,10,15]$ & $[3,7,16,v_3]$ & $[11,12,13,v_{12}]$ & $[0,10,15,v_3]$ & $[0,2,v_2,v_{23}]$ & $[11,19,v_2,v_{12}]$ & $[13,17,21,v_1]$\\
			$[5,12,15,16]$ & $[10,13,21,v_1]$ & $[0,6,15,17]$ & $[8,16,20,21]$ & $[0,1,17,v_2]$ & $[1,8,16,20]$ & $[5,11,12,v_{13}]$ & $[2,6,20,v_{12}]$\\
			$[0,2,6,v_{23}]$ & $[14,20,v_1,v_{12}]$ & $[2,19,v_3,v_{12}]$ & $[0,3,11,v_1]$ & $[1,10,20,v_2]$ & $[5,10,12,15]$ & $[5,17,v_2,v_{13}]$ & $[2,7,12,15]$\\
			$[1,11,13,v_1]$ & $[5,11,18,v_{13}]$ & $[2,7,18,v_{23}]$ & $[0,8,16,v_{12}]$ & $[1,11,19,v_2]$ & $[8,11,19,v_{12}]$ & $[4,16,v_2,v_{13}]$  & $[5,17,v_1,v_{12}]$\\
			$[5,17,21,v_2]$ & $[7,10,18,v_{13}]$ & $[1,10,19,v_2]$ & $[2,5,v_3,v_{12}]$ & $[0,3,16,v_3]$ & $[7,10,12,v_{12}]$ & $[0,17,v_3,v_{12}]$ & $[0,16,v_3,v_{12}]$\\
			$[7,11,v_3,v_{13}]$ & $[17,18,19,v_{23}]$ & $[11,13,v_2,v_{12}]$ & $[2,7,15,v_1]$ & $[2,5,6,v_{12}]$ & $[2,7,v_1,v_{23}]$ & $[8,16,19,v_{12}]$ & $[4,16,v_2,v_{23}]$\\
			$[8,11,19,21]$ & $[10,20,21,v_3]$ & $[8,11,20,v_{12}]$ & $[5,6,15,16]$ & $[2,7,12,v_{12}]$ & $[6,16,19,21]$ & $[15,16,19,v_3]$ & $[1,14,18,20]$\\
			$[1,17,20,v_1]$ & $[2,9,15,v_1]$ & $[2,12,20,v_{13}]$ & $[10,19,v_2,v_{12}]$ & $[16,20,21,v_2]$ & $[0,1,16,v_2]$ & $[1,16,20,v_2]$ & $[0,11,21,v_1]$\\
			$[5,17,21,v_1]$ & $[2,19,v_3,v_{23}]$ & $[7,11,v_1,v_{23}]$ & $[10,12,13,v_{23}]$ & $[10,13,21,v_2]$ & $[15,17,19,v_3]$ & $[0,10,v_1,v_{23}]$ & $[1,13,17,v_1]$\\
			$[0,16,v_2,v_{23}]$ & $[2,4,v_1,v_{23}]$ & $[14,18,20,v_{13}]$ & $[7,15,16,v_1]$ & $[17,20,v_1,v_{12}]$ & $[9,12,16,v_{13}]$ & $[7,11,18,v_{23}]$ & $[11,12,20,v_{12}]$\\
			$[5,17,v_3,v_{12}]$ & $[11,20,21,v_3]$ & $[2,18,19,v_{23}]$ & $[7,10,15,v_3]$ & $[7,10,12,15]$ & $[5,6,14,21]$ & $[11,12,20,v_{13}]$ & $[0,3,16,v_1]$\\
			$[2,9,12,15]$ & $[1,11,13,v_2]$ & $[2,9,12,v_{13}]$ & $[7,10,18,v_{12}]$ & $[2,9,v_1,v_{13}]$ & $[6,15,16,19]$ & $[2,5,v_3,v_{23}]$ & $[5,17,18,v_{13}]$\\
			$[0,3,11,v_3]$ & $[10,13,v_2,v_{12}]$ & $[7,10,v_3,v_{13}]$ & $[0,6,17,v_{13}]$ & $[2,7,18,v_{12}]$ & $[10,13,v_1,v_{23}]$ & $[16,19,v_3,v_{12}]$ & $[6,15,17,19]$\\
			\hline\hline
		\end{tabular}}
	\caption{The list of facets of $\Delta^{3T}_{28}$. The 6 vertices of the double-trefoil knot are labeled $v_F$, according to the face $F$ they correspond to in the barycentric subdivision.}
	\label{facets_triple_knot}
\end{table}	
\subsection{Normal $3$-pseudomanifolds}
We conclude this section with a broader class of triangulations, which includes combinatorial manifolds. 
\begin{definition}
	Let $\Delta$ be a pure $d$-dimensional pseudomanifold. $\Delta$ is a \emph{normal $d$-pseudomanifold} if the link of each face of dimension at most $(d-2)$ is connected. 
\end{definition}
\begin{table}[h!]
	\[\begin{array}{l|l l l|l}
	\text{Hom. type} & \text{Min }f(\Delta) &  \quad\quad\quad\quad & \text{Hom. type} & \text{Min }f(\Delta) \\\cline{1-2}\cline{4-5}
	8,0,1,0	&	(1, 12, 51, 80, 40)	&		&	2,4,1,2	&	(1, 24, 155, 272, 136)	\\
	8,0,0,1	&	(1, 14, 61, 96, 48)	&		&	2,2,2,3	&	(1, 24, 155, 272, 136)	\\
	7,2,0,0	&	(1, 11, 42, 64, 32)	&		&	2,2,1,4	&	(1, 28, 194, 344, 172)	\\
	7,0,2,0	&	(1, 11, 45, 72, 36)	&		&	2,2,0,5	&	(1, 27, 181, 320, 160)	\\
	7,0,0,2	&	(1, 13, 55, 88, 44)	&		&	1,8,0,0	&	(1, 25, 155, 268, 134)	\\
	6,2,0,1	&	(1, 15, 69, 112, 56)	&		&	1,6,1,1	&	(1, 28, 193, 340, 170)	\\
	6,0,3,0	&	(1, 15, 72, 120, 60)	&		&	1,4,4,0	&	(1, 28, 190, 336, 168)	\\
	6,0,1,2	&	(1, 15, 76, 128, 64)	&		&	1,4,2,2	&	(1, 28, 188, 332, 166)	\\
	6,0,0,3	&	(1, 16, 83, 140, 70)	&		&	1,4,1,3	&	(1, 29, 193, 340, 170)	\\
	5,4,0,0	&	(1, 16, 78, 128, 64)	&		&	1,4,0,4	&	(1, 28, 194, 344, 172)	\\
	5,2,1,1	&	(1, 17, 94, 160, 80)	&		&	1,2,4,2	&	(1, 31, 214, 380, 190)	\\
	5,2,0,2	&	(1, 18, 91, 152, 76)	&		&	1,2,3,3	&	(1, 30, 205, 364, 182)	\\
	5,0,2,2	&   (1, 19, 109, 188, 94)	&		&	1,2,2,4	&	(1, 30, 207, 368, 184)	\\
	5,0,0,4	&	(1, 18, 110, 192, 96)	&		&	1,2,0,6	&	(1, 31, 218, 388, 194)	\\
	4,4,1,0	&	(1, 19, 104, 176, 88)	&		&	1,0,8,0	&	(1, 30, 214, 384, 192)	\\
	4,4,0,1	&	(1, 19, 104, 176, 88)	&		&	1,0,4,4	&	(1, 33, 235, 420, 210)	\\
	4,2,2,1	&	(1, 20, 120, 208, 104)	&		&	0,8,1,0	&	(1, 31, 196, 340, 170)	\\
	4,2,1,2	&	(1, 20, 120, 208, 104)	&		&	0,6,0,3	&	(1, 32, 214, 376, 188)	\\
	4,2,0,3	&	(1, 20, 122, 212, 106)	&		&	0,4,4,1	&	(1, 32, 217, 384, 192)	\\
	4,0,5,0	&	(1, 22, 139, 244, 122)	&		&	0,4,3,2	&   (1, 31, 210, 372, 186)	\\
	4,0,1,4	&	(1, 23, 144, 252, 126)	&		&	0,4,0,5	&	(1, 33, 240, 428, 214)	\\
	3,6,0,0	&	(1, 21, 114, 192, 96)	&		&	0,2,4,3	&	(1, 38, 290, 520, 260)	\\
	3,4,2,0	&	(1, 22, 128, 220, 110)	&		&	0,2,3,4	&	(1, 32, 244, 440, 220)	\\
	3,4,0,2_a&	(1, 22, 134, 232, 116)	&		&	0,2,2,5	&	(1, 37, 263, 468, 234)	\\
	3,4,0,2_b&	(1, 22, 138, 240, 120)	&		&	0,0,9,0	&	(1, 41, 296, 528, 264)	\\
	3,2,2,2	&	(1, 24, 159, 280, 140)	&		&	0,0,5,4	&	(1, 36, 281, 508, 254)	\\
	3,2,1,3	&	(1, 23, 144, 252, 126)	&		&	0,0,3,6	&	(1, 33, 260, 472, 236)	\\
	2,6,0,1	&	(1, 23, 139, 240, 120)	&		&	0,0,1,8	&	(1, 36, 265, 476, 238)	\\
	2,4,2,1	&	(1, 26, 179, 316, 158)	&       &	0,0,0,9	&	(1, 36, 265, 476, 238)	\\

	\end{array}
	\]
	\caption{A table reporting the minimal $f$-vector found of a balanced triangulation of some normal $3$-pseudomanifolds with singularities.}
	\label{tab:table4}
\end{table}
\noindent
For $d=2$ this class of simplicial complexes coincides with that of triangulated surfaces. Moreover, since the vertex links of a normal $d$-pseudomanifold are normal $(d-1)$-pseudomanifolds, it follows that for a balanced normal $3$-pseudomanifold the vertex links are balanced triangulated surfaces. In this section we report small balanced normal $3$-pseudomanifolds which are not combinatorial $3$-manifolds, obtained applying cross-flips to the barycentric subdivision of the complexes on 9 vertices enumerated by Akhmejanov \cite{Pseudo} modifying a computer program by Sulanke. It is very important to observe that since \Cref{cross-flipsIKN} holds only for combinatorial manifolds there is no connectivity result for the set of balanced normal $3$-pseudomanifolds and hence it might be the case that the barycentric subdivision we start from and the balanced vertex-minimal triangulation lie in a different connected component (of the cross-flip graph). In \Cref{tab:table4} (as well as in \cite{GitRepo}) we simply exhibit the $f$-vector of the complex with minimum number of vertices among those that our program returned after a fixed number of iterations. Still cross-flips clearly preserve the homeomorphism type, so the complexes whose $f$-vectors appear in \Cref{tab:table4} are indeed balanced triangulations of the spaces in \cite{Pseudo}. The first column of \Cref{tab:table4} reports the number of vertices whose link is homeomorphic to $S^2$, $\mathbb{R}\mathbf{P}^{2}$, $S^1\times S^1$ and $\mathbb{R}\mathbf{P}^{2}\#\mathbb{R}\mathbf{P}^{2}$ respectively, since those are the only homeomorphism types that can appear as vertex links of (non-balanced) normal $3$-pseudomanifolds with up to 9 vertices. Note that, except for the case $(3,4,0,2)$, this determines the homeomorphism type. From this numbers one can easily infer the singularity type of the corresponding balanced triangulation, since the number of vertex links not homeomorphic to $S^2$ does not change in the reduction process. 
\begin{remark}
	Observe that reducing the barycentric subdivisions of the normal pseudomanifolds of type $(7,2,0,0)$, $(7,0,2,0)$ and $(7,0,0,2)$ the program returned the suspension of $\Delta^{\mathbb{R}\mathbf{P}^{2}}_{9}$, of the balanced vertex-minimal triangulation of $S^1\times S^1$ and of a balanced vertex-minimal triangulation of $\mathbb{R}\mathbf{P}^{2}\#\mathbb{R}\mathbf{P}^{2}$ respectively.
\end{remark}

\section*{Acknowledgements}
\noindent
I would like to thank Martina Juhnke-Kubitzke for her constant support and precious help. I would also like to thank Isabella Novik and Hailun Zheng for helpful comments on a preliminary version of this paper and Bruno Benedetti for pointing out the results in \cite{BeLu}. Finally I would like to thank Francesco Croce for his careful proofreading.
\bibliographystyle{alpha}
\bibliography{bibliography}

\begin{thebibliography}{JKMNS18}

\bibitem[Akh]{Pseudo}
T.~Akhmejanov.
\newblock {Personal page}.
\newblock
  \url{https://http://pi.math.cornell.edu/~takhmejanov/pseudoManifolds.html}.

\bibitem[Bj{\"{o}}95]{Bj95}
A.~Bj{\"{o}}rner.
\newblock Topological methods.
\newblock In {\em Handbook of combinatorics, {V}ol. 1, 2}, pages 1819--1872.
  Elsevier Sci. B. V., Amsterdam, 1995.

\bibitem[BL00]{BjLu}
A.~Bj\"{o}rner and F.~H. Lutz.
\newblock Simplicial manifolds, bistellar flips and a 16-vertex triangulation
  of the {P}oincar\'{e} homology 3-sphere.
\newblock {\em Experiment. Math.}, 9(2):275--289, 2000.

\bibitem[BL13]{BeLu}
B.~Benedetti and F.~H. Lutz.
\newblock Knots in collapsible and non-collapsible balls.
\newblock {\em Electron. J. Combin.}, 20(3):Paper 31, 29, 2013.

\bibitem[BL14]{BeLuRandom}
B.~Benedetti and F.~H. Lutz.
\newblock Random discrete {M}orse theory and a new library of triangulations.
\newblock {\em Exp. Math.}, 23(1):66--94, 2014.

\bibitem[CFSV04]{CoFoSaVe}
L.~P. Cordella, P.~Foggia, C.~Sansone, and M.~Vento.
\newblock A (sub)graph isomorphism algorithm for matching large graphs.
\newblock {\em IEEE Transactions on Pattern Analysis and Machine Intelligence},
  26(10):1367--1372, Oct 2004.

\bibitem[EH01]{EhHa}
R.~Ehrenborg and M.~Hachimori.
\newblock Non-constructible complexes and the bridge index.
\newblock {\em European J. Combin.}, 22(4):475--489, 2001.

\bibitem[HZ00]{HaZi}
M.~Hachimori and G.~M. Ziegler.
\newblock Decompositons of simplicial balls and spheres with knots consisting
  of few edges.
\newblock {\em Math. Z.}, 235(1):159--171, 2000.

\bibitem[IKN17]{IKN}
I.~Izmestiev, S.~Klee, and I.~Novik.
\newblock Simplicial moves on balanced complexes.
\newblock {\em Adv. Math.}, 320:82--114, 2017.

\bibitem[JKM18]{JKM}
M.~Juhnke-Kubitzke and S.~Murai.
\newblock Balanced generalized lower bound inequality for simplicial polytopes.
\newblock {\em Selecta Math. (N.S.)}, 24(2):1677--1689, 2018.

\bibitem[JKMNS18]{Juhnke:Murai:Novik:Sawaske}
M.~Juhnke-Kubitzke, S.~Murai, I.~Novik, and C.~Sawaske.
\newblock A generalized lower bound theorem for balanced manifolds.
\newblock {\em Math. Z.}, 289(3-4):921--942, 2018.

\bibitem[JV18]{2018arXiv180406270J}
M.~{Juhnke-Kubitzke} and L.~{Venturello}.
\newblock {Balanced shellings and moves on balanced manifolds}.
\newblock {\em ArXiv e-prints}, April 2018.
\newblock \url{https://arxiv.org/abs/1804.06270}.

\bibitem[KN16]{KN}
S.~Klee and I.~Novik.
\newblock Lower bound theorems and a generalized lower bound conjecture for
  balanced simplicial complexes.
\newblock {\em Mathematika}, 62(2):441--477, 2016.

\bibitem[Lut]{Lu}
F.~H. Lutz.
\newblock The {M}anifold {P}age.
\newblock \url{http://page.math.tu-berlin.de/~lutz/stellar/}.

\bibitem[Lut99]{LuThesis}
F.~H. Lutz.
\newblock {\em Triangulated manifolds with few vertices and vertex-transitive
  group actions}.
\newblock Berichte aus der Mathematik. [Reports from Mathematics]. Verlag
  Shaker, Aachen, 1999.
\newblock Dissertation, Technischen Universit\"{a}t Berlin, Berlin, 1999.

\bibitem[MS18]{MuSu}
S.~Murai and Y.~Suzuki.
\newblock Balanced subdivisions and flips on surfaces.
\newblock {\em Proc. Amer. Math. Soc.}, 146(3):939--951, 2018.

\bibitem[Mur15]{MUR}
S.~Murai.
\newblock Tight combinatorial manifolds and graded {B}etti numbers.
\newblock {\em Collect. Math.}, 66(3):367--386, 2015.

\bibitem[Sta79]{St79}
R.~P. Stanley.
\newblock Balanced {C}ohen-{M}acaulay complexes.
\newblock {\em Trans. Amer. Math. Soc.}, 249(1):139--157, 1979.

\bibitem[Sta96]{Stanley-greenBook}
R.~P. Stanley.
\newblock {\em Combinatorics and commutative algebra}, volume~41 of {\em
  Progress in Mathematics}.
\newblock Birkh\"auser Boston, Inc., Boston, MA, second edition, 1996.

\bibitem[{The}17]{sagemath}
{The Sage Developers}.
\newblock {\em {S}ageMath, the {S}age {M}athematics {S}oftware {S}ystem
  ({V}ersion 8.1)}, 2017.
\newblock \url{http://www.sagemath.org}.

\bibitem[Ven]{GitRepo}
L.~Venturello.
\newblock {Git Hub Repository}.
\newblock
  \url{https://github.com/LorenzoVenturello/Cross-flips_and_balanced_library}.

\bibitem[{Zhe}16a]{Zh17}
H.~{Zheng}.
\newblock {Ear Decompostion and Balanced 2-neighborly Simplicial Manifolds}.
\newblock {\em ArXiv e-prints}, December 2016.
\newblock \url{https://arxiv.org/abs/1612.03512}.

\bibitem[Zhe16b]{Zh16}
H.~Zheng.
\newblock Minimal balanced triangulations of sphere bundles over the circle.
\newblock {\em SIAM J. Discrete Math.}, 30(2):1259--1268, 2016.

\end{thebibliography}
\end{document}